\documentclass[a4paper,reqno, 10pt]{amsart}

\usepackage{amsmath,amssymb,amsfonts,amsthm,mathrsfs}

\usepackage{verbatim,wasysym,cite}

\usepackage[latin1]{inputenc}
\usepackage{microtype}
\usepackage{color,enumitem,graphicx}
\usepackage[colorlinks=true,urlcolor=blue, citecolor=red,linkcolor=blue,
linktocpage,pdfpagelabels, bookmarksnumbered,bookmarksopen]{hyperref}
\usepackage[hyperpageref]{backref}
\usepackage[english]{babel}
\usepackage[symbol]{footmisc}
\renewcommand{\epsilon}{{\varepsilon}}

\numberwithin{equation}{section}
\newtheorem{theorem}{Theorem}[section]
\newtheorem{lemma}[theorem]{Lemma}
\newtheorem{remark}[theorem]{Remark}
\newtheorem{definition}[theorem]{Definition}
\newtheorem{proposition}[theorem]{Proposition}

\newcommand{\R}{\mathbb R}
\newcommand{\N}{\mathbb N}
\newcommand{\T}{\mathbb T}

\def\({\left(}
\def\){\right)}
\def\<{\left\langle}
\def\>{\right\rangle}

\def\d{{\partial}}
\def\eps{\varepsilon}
\def\si{\sigma}

\DeclareMathOperator{\IM}{Im}

\def\Tend#1#2{\mathop{\longrightarrow}\limits_{#1\rightarrow#2}}

\begin{document}

\title[NLS under partial harmonic confinement]
{Global dynamics below the ground states for NLS under partial harmonic confinement}

\author[Alex H. Ardila]{Alex H. Ardila}
\address{Universidade Federal de Minas Gerais\\ ICEx-UFMG\\ CEP
  30123-970\\ MG, Brazil} 
\email{ardila@impa.br}
\author[R\'emi Carles]{R\'emi Carles}
\address{Univ Rennes, CNRS\\ IRMAR - UMR 6625\\ F-35000 Rennes\\ France}
\email{Remi.Carles@math.cnrs.fr}

\begin{abstract}
We are concerned with the global behavior of the
solutions of the focusing mass supercritical nonlinear Schr\"odinger
equation under partial harmonic confinement. We establish a necessary
and sufficient condition on the initial data below the ground states
to determine the global behavior (blow-up/scattering) of the
solution.  Our proof of scattering is based on the variational
characterization of the ground states,  localized virial estimates,
linear profile decomposition and nonlinear profiles. 
\end{abstract}

\thanks{RC is supported by Rennes M\'etropole through its AIS program. }

\subjclass[2010]{35Q55, 37K45, 35P25}
\keywords{NLS;  ground states; global existence; blow-up; scattering.}

\maketitle




\medskip

\section{Introduction}
\label{sec:intro}

In this paper we study the initial-value problem for the nonlinear Schr\"odinger
equation under partial harmonic confinement
\begin{equation}\label{GP}
 \begin{cases} 
i\partial_{t}u=H u+\lambda |u|^{2\sigma}u,\quad x \in\mathbb{R}^{d},\quad t\in\mathbb{R},\,\,\\
u(0,x)=u_{0}(x),
\end{cases} 
\end{equation}
where $u: \mathbb{R}\times\mathbb{R}^{d}\rightarrow \mathbb{C}$,
$\lambda\in \{-1,+1\}$, 
$d\geq2$  and $0<\sigma<\tfrac{2}{d-2}$.  The operator $H$ is
defined as
\[H:=-\Delta_{y}+|y|^{2}-\Delta_{z},  \quad
  x=(y,z)\in  \mathbb{R}^{n}\times  \mathbb{R}^{d-n},\]
where $1\leq n\leq d-1$. The equation \eqref{GP} arises 
in various branches of physics, such as the Bose-Einstein condensates or the propagation of mutually incoherent wave packets in nonlinear optics.  For more details we refer to \cite{JossePomea2001}. 

As recalled briefly in Section~\ref{S:es}, the Cauchy problem for
\eqref{GP} is locally well-posed in the energy space\footnote{The
  notation $B_1$ is borrowed from \cite{BenCasteMeht2008}, for
  consistency in future references.} 
\begin{equation*} 
{B}_{1}=\big\{u\in H^{1}(\mathbb{R}^{d}; \mathbb{C}): \|y u\|^{2}_{L^{2}}=\int_{\mathbb{R}^{d}}|y|^{2}|u(x)|^{2}dx<\infty  \big\},
\end{equation*}
equipped with the norm
\begin{equation*} 
\|u\|^{2}_{{B}_{1}}=\<u,Hu\>=\|\nabla_{x} u\|^{2}_{L^{2}}+\|yu\|^{2}_{L^{2}}+\|u\|^{2}_{L^{2}}.
\end{equation*}
In particular, the linear propagator $e^{-itH}$ preserves the
$B_1$-norm. 
We can use a contraction mapping technique based on Strichartz
estimates to show that \eqref{GP} is locally well-posed in ${B}_{1}$
(see Lemma~\ref{lem:LWP}): for any $u_{0}\in B_{1}$ there exists a
unique maximal solution $u\in C((-T_-,T_+); B_{1})$ of \eqref{GP}, $T_\pm\in(0, \infty]$.  Furthermore, the solution $u$  enjoys the conservation of energy, momentum and mass,
\begin{equation}\label{Mome}
E(u(t))=E(u_{0}), \quad G(u(t))=G(u_{0}), \quad M(u(t))=M(u_{0}),\quad
\forall t\in (-T_-,T_+),
\end{equation}
where $E$, $M$ and $G$ are defined as
\[ E(u)=\frac{1}{2}\int_{\mathbb{R}^{d}}|\nabla_{x} u|^{2}dx+\frac{1}{2}\int_{\mathbb{R}^{d}}|y|^{2}|u|^{2}dx+\frac{\lambda}{2\sigma+2}\int_{\mathbb{R}^{d}}|u|^{2\sigma+2}dx, \]
and
\begin{equation}\label{Moment}
G(u)=\mbox{Im}\int_{\mathbb{R}^{d}}\overline{u}\nabla_{z}udx,\quad
  M(u)=\int_{\mathbb{R}^{d}}|u|^{2}dx.
\end{equation}
We recall the definitions of scattering and blow-up in the framework
of the energy space $B_1$.
\begin{definition}
Let $u$ be a solution of the Cauchy problem \eqref{GP} on the maximal existence time interval $(-T_{-}, T_{+})$. We say that the solution $u$ scatters in $B_{1}$ (both forward and backward time) if $T_{\pm}=\infty$ and there exist $\psi^{\pm}\in B_{1}$ such that
\begin{equation*}
\|u(t)-e^{-itH}\psi^{\pm}\|_{B_{1}}=\|e^{itH}u(t)-\psi^{\pm}\|_{B_{1}} \rightarrow0 \quad \text{as $t\rightarrow \pm \infty$.}
\end{equation*}
On the other hand, if $T_{+}<\infty$ (resp. $T_{-}<\infty$), we say
that the solution $u$ blows up in positive time (resp. negative
time). In the case $T_+<\infty$, this corresponds to the property
\begin{equation*}
  \|\nabla_x u(t)\|_{L^2(\R^d)}\Tend t {T_+} \infty. 
\end{equation*}
\end{definition}
We refer to the proof of Lemma~\ref{lem:LWP} below to see why the
momentum does not appear in the blow-up characterization. In
\cite{AntoCarSilva2015}, scattering was considered in the conformal
space
\begin{equation*}
  \Sigma= B_1\cap \{f;\ x\mapsto |z| f(x)\in L^2(\R^d)\} =
  H^1(\R^d)\cap \{ f;\ x\mapsto |x| f(x)\in L^2(\R^d)\} ,
\end{equation*}
which is of course smaller than $B_1$. In the present paper, we
investigate the large time behavior of the solution to \eqref{GP} in
$B_1$, both in the focusing ($\lambda=-1$) and in the defocusing
($\lambda=1$) case. As a preliminary, we state a result concerning the
small data case. 
\begin{proposition}\label{prop:small-general}
 Suppose
  $\tfrac{2}{d-n}\le \si<\tfrac{2}{d-2}$ and $\lambda\in
  \{-1,+1\}$. There exists $\nu>0$ 
  such that if $\|u_0\|_{B_1}\le \nu$, then the solution to
  \eqref{GP} is global in time ($T_\pm =\infty$) and scatters in
  $B_1$. 
\end{proposition}
This proposition follows directly from Lemma~\ref{ssp}
below. We note that in \cite{AntoCarSilva2015}, for the similar
statement in the smaller space $\Sigma$, the lower bound on $\si$ was
$\si>\tfrac{d}{d+2}\tfrac{2}{d-n}$ (see
\cite[Theorem~1.5]{AntoCarSilva2015}). In terms of the variable $y\in
\R^n$, confinement 
prevents complete dispersion. On the other hand, in the variable $z\in
\R^{d-n}$, we benefit from the usual dispersion for the Schr\"odinger
equation posed on $\R^{d-n}$. In other words, scattering is expected
\emph{somehow} as if we considered
\begin{equation*}
  i\d_t v = -\Delta_z v +\lambda|v|^{2\si} v,\quad z\in \R^{d-n},
\end{equation*}
and the above lemma is the counterpart of small data scattering in
$H^1(\R^{d-n})$ for $L^2$-critical or supercritical nonlinearities,
and the presence of the extra variable $y$ reads in the upper bound
$\si<\tfrac{2}{d-2}$, to make the nonlinearity
energy-subcritical. For large data, global existence and some blow-up results have
been considered in \cite{Carles2011}. 
Moreover,  scattering for 
\eqref{GP}, for some $\sigma$, $d$ and $n$, was studied in  
\cite{AntoCarSilva2015, CarlesGallo2015,HaniThoman2016}.
\bigbreak

Consider the focusing case $\lambda=-1$, which is the core of this paper.
In the case $0<\sigma<2/d$ the Cauchy problem \eqref{GP} is globally
well-posed, regardless of the sign of $\lambda$. Moreover,  for small 
initial data the solution can be 
extended to a global one in the case $2/d<\sigma<2/(d-2)$.   The issue of existence, stability and instability of
standing waves has been studied in \cite{BEBOJEVI2017,OhtaPartial2018, STW}.

Introduce the following nonlinear elliptic problem
\begin{equation}\label{Ep}
  H \varphi+\varphi-|\varphi|^{2\sigma}\varphi=0,\quad
\varphi\in B_{1}\setminus \left\{0\right\}.
\end{equation}
We recall that a non-trivial solution $Q$ to  \eqref{Ep} is said to be
the ground state solution, if it has some minimal action among all
solutions of the elliptic problem \eqref{Ep}, i.e. 
 \begin{equation}\label{PM1}
S(Q)=\inf\left\{S(\varphi):\, \text{ $\varphi$ is a solution of \eqref{Ep}}\right\},
\end{equation}
where the action functional $S$ is defined by 
\begin{equation*}
S(u):=\frac{1}{2}\|\nabla_{x} u\|^{2}_{L^{2}}+\frac{1}{2}\|yu\|^{2}_{L^{2}}+\frac{1}{2}\|u\|^{2}_{L^{2}}-\frac{1}{2\sigma+2}\|u\|^{2\sigma+2}_{L^{\sigma+2}}.
\end{equation*}
In Lemma~\ref{L1} we obtain the existence of at least one ground state solution (see also Remark~\ref{Rgs}). 
\begin{remark}
  We could also consider, for any $\omega>0$,
  \begin{equation*}
     H \varphi+\omega\varphi-|\varphi|^{2\sigma}\varphi=0,\quad
\varphi\in B_{1}\setminus \left\{0\right\},
\end{equation*}
up to adapting the notations throughout the paper. We consider the
case $\omega=1$ for simplicity. 
\end{remark}
Our main result consists in establishing a necessary and
sufficient condition on the initial data below the ground state $Q$ to
determine the global behavior (blow-up/scattering) of the solution. As
recalled above, when scattering occurs, it is reminiscent of the
nonlinear Schr\"odinger equation without potential, posed on
$\R^{d-n}$. With this in mind,  we define the following functional of
class $C^{2}$ on $B_{1}$,
\begin{equation}\label{Fp}
P(u)=\frac{2}{d-n}\|\nabla_{z}u\|^{2}_{L^{2}}-\frac{\sigma}{\sigma+1}\|u\|^{2\sigma+2}_{L^{\sigma+2}},
\end{equation}
and we define the following subsets in ${B}^{1}$, 
\begin{align*}
\mathcal{K}^{+}&=\bigl\{ \varphi\in {B}_{1}: S(\varphi)<S(Q),\quad P(\varphi)\geq 0\bigl\}, \\
\mathcal{K}^{-}&=\bigl\{ \varphi\in {B}_{1}: S(\varphi)<S(Q),\quad P(\varphi)<0\bigl\}.  
\end{align*}
By a scaling argument, it is not difficult to show that $\mathcal{K}^{\pm}\neq \emptyset$. 
In our main result, we will show  that the sets
$\mathcal{K}^{+}$ and $\mathcal{K}^{-}$ are
invariant under the flow generated  by the  equation
\eqref{GP}. Moreover,  we obtain a sharp criterion between blow-up and
scattering for \eqref{GP} in terms of the functional $P$ given by
\eqref{Fp}. In the case of a full confinement ($n=d$), such results
were initiated in\cite{Zhang2005,ShuZhang2006}. Of course, in the
absence of fully dispersive direction, the dichotomy concerns global
existence \emph{vs.} blow-up, and scattering cannot hold. The proof
of scattering properties represents a large part of the present
paper. 

The assumption $\sigma>\tfrac{2}{d-n}$ is needed to prove the
Lemmas~\ref{L1} and \ref{L3} (existence and characterization of the
ground states) and the profile decomposition result (see
Proposition~\ref{Ls3}).  Thus, in the case $\lambda=-1$, we assume
\[\frac{2}{d-n}<\sigma<\frac{2}{d-2}.\]
This condition implies that $n=1$ in the statement below, a condition
which is reminiscent of \cite{TzVi2016}, where a partial
one-dimensional \emph{geometrical} confinement is considered ($y\in
\T$). Also, a step of our proof requires the extra property $\si\ge
\tfrac{1}{2}$, and so we restrict to dimensions $2\le d\le 5$. 
\begin{theorem}\label{Th1}
  Let $\lambda=-1$, $n=1$, $\si\ge \tfrac{1}{2}$ with
  $\tfrac{2}{d-1}<\si<\tfrac{2}{d-2}$, and $u_0\in B_1$. 
Let $u\in C(I;B_1)$ be the corresponding solution of \eqref{GP}  with initial data $u_{0}$ and lifespan $I=(T_{-}, T_{+})$. \\
(i) If $u_{0}\in \mathcal{K}^{+}$, then the corresponding solution $u(t)$ exists globally and scatters.\\
(ii) If $u_{0}\in \mathcal{K}^{-}$, then one of the following two cases occurs:
 \begin{enumerate}
    	\item  The solution blows up in positive time, i.e.,
          $T_{+}<\infty$ and
          \[\lim_{t\rightarrow T_+}\|\nabla_{x} u(t)\|^{2}_{L^{2}}=\infty.\]
      \item  The solution blows up at infinite positive time,  i.e.,  $T_{+}=\infty$ and there exists a sequence $\left\{t_{k}\right\}$ such that $t_{k}\rightarrow \infty$ and $\lim_{t_{k}\rightarrow \infty}\|\nabla_{x} u(t_{k})\|^{2}_{L^{2}}=\infty$.
		\end{enumerate}
An analogous statement holds for negative time.
\end{theorem}
\begin{remark}
  We note that if the initial datum satisfies $u_{0}\in
  \mathcal{K}^{-}$ and $xu_{0}\in L^{2}(\mathbb{R}^{d})$ (that is,
  $u_0\in \Sigma$),
  then the corresponding solution blows up in finite time (see
  \eqref{LV3} below for more details, with $R=\infty$). In particular, the condition $P(u)\geq 0$ in Theorem  \ref{Th1} is sharp for global existence.
\end{remark}
The proof of the scattering result is based on the
concentration/compactness and rigidity  argument of Kenig-Merle \cite{
  KenigMerle2006}. In \cite{DuyHolmerRoude2008},
Duyckaerts-Holmer-Roudenko studied  \eqref{GP} with $d=3$,  
$\sigma=1$, without harmonic potential, and proved that if $u_{0}\in
H^{1}(\mathbb{R}^{3})$ satisfies (see also
\cite{HolmerRoudenko2008} in the radial case)
\begin{equation*}
M(u_{0})E(u_{0})<M(Q)E(Q), \quad \|u_{0}\|_{L^{2}}\|\nabla u_{0}\|_{L^{2}}<\|Q\|_{L^{2}}\|\nabla Q\|_{L^{2}},
\end{equation*}
then the corresponding solution exists globally and scatters in
$H^{1}(\mathbb{R}^{3})$, where $Q$ is the ground state of the equation
\eqref{Ep}. However, it seems that the method developed in
\cite{DuyHolmerRoude2008,HolmerRoudenko2008} cannot be applied to  \eqref{GP} with harmonic potential. The main difficulty concerning 
\eqref{GP} is clearly the presence of the partial harmonic
confinement. In particular, we cannot apply scaling techniques to
obtain the critical element (see the proof of Proposition~5.4 in
\cite{HolmerRoudenko2008}). To overcome this problem, we use a
variational approach based on the work of Ibrahim-Masmoudi-Nakanishi
\cite{IbraMasmNaka2011} (see also \cite{IkeaInu2017}).  We mention the
works of Ikea-Inu \cite{ IkeaInu2017} and Guo-Wang-Yao
\cite{GuoWangYao2018} who also obtained analogous result to Theorem
\ref{Th1} for the focusing NLS equation with a potential. The proof of
the blow-up result is based on the techniques developed by Du-Wu-Zhang
\cite{DuWuZhang2016}. 

It is worth mentioning that Fang-Xie-Cazenave \cite{FangXieCaze2011}
and Akahor-Nawa \cite{AkahoriNawa2013} extended the results in
Holmer-Roudenko \cite{HolmerRoudenko2008} and
Duyckaerts-Holmer-Roudenko \cite{DuyHolmerRoude2008} in terms of
dimension and power. Concerning the scattering theory with a smooth
short range potential
in the energy-subcritical case, we  refer to
\cite{Carles2016,CassanoDAncona2016,Hong2016,Lafontaine2016}; see also
\cite{BanicaVisciglia2016,KillipMurphyVisanZheng}  for scattering
theory with a singular potential in the energy-subcritical case. For other
results, see e.g. 
\cite{FarahGuzman2017, DodsonMurphy2017,
  BellazziniForcella2019}, and \cite{HaniThoman2016} in the case of a
partial confinement leading to long range scattering for small data.

\begin{remark}
The tools that we use also yield scattering 
results in the defocusing case $\lambda=+1$.
For $d\ge 2$, $n=1$, and $\si\ge \tfrac{1}{2}$ with
 $   \tfrac{2}{d-1}< \si<\tfrac{2}{d-2}$, consider $u_0\in B_1$ and
 $u\in C(\R;B_1)$ the solution to 
  \begin{equation*}
    i\d_t u = H u +|u|^{2\si}u\quad;\quad u_{\mid t=0}=u_0. 
  \end{equation*}
  Then $u$ scatters in $B_1$. As pointed out in
  \cite[Section~7]{DuyHolmerRoude2008} in the 
case of the 3D cubic Schr\"odinger equation without potential, the
proof is essentially the 
same as for scattering in the focusing case (Theorem~\ref{Th1}). Also,
in this defocusing case, we simply recover
\cite[Theorem~1.5]{CarlesGallo2015}, based on Morawetz estimates, where the assumption $\si\ge
\tfrac{1}{2}$ was not needed. 
\end{remark}

\subsection*{Organization of the paper}
 In the next section we introduce Strichartz estimates  specific to the
 present context, and show that  a specific norm suffices to
 ensure scattering.  In Section~\ref{S:1}, we
 show variational estimates, which will be key to obtain blow-up and
 scattering results in the focusing case. In Section~\ref{S:2}, we
 show the blow-up results 
 and the global part of Theorem~\ref{Th1} (i). Finally, in
 Section~\ref{SC1} we prove the scattering part of Theorem~\ref{Th1}.

\subsection*{Notations}
We summarize the notation used throughout the paper: $\mathbb{Z}$
denotes the set of all  integers.  We will use $A\lesssim B$
(resp. $A\gtrsim B$) for inequalities of type $A\leq  CB$
(resp. $A\geq CB$), where $C$ is a positive constant.  If both the
relations hold true, we write $A\sim B$.
We denote by $\text{NLS}(t)u_{0}$ the solution of the IVP \eqref{GP}
with initial data $u_{0}$.\\ 
For $1\leq p\leq\infty$, we denote its conjugate by
$p^{\prime}=\tfrac{p}{p-1}$. Moreover, $L^{p}=L^{p}(\mathbb{R}^{d};
\mathbb{C})$ are the classical Lebesgue spaces.  The scale of harmonic
(partial) Sobolev spaces is defined as follows, see
\cite{BenCasteMeht2008}: for $s\geq0$ 
\begin{equation*}
B_{s}=B_{s}(\mathbb{R}^{d})=\left\{ u\in L^{2}(\mathbb{R}^{d}): H^{s/2}u\in L^{2}(\mathbb{R}^{d}) \right\}
\end{equation*}
endowed with the natural norm  denoted by $\|\cdot\|_{B_{s}}$, and up to equivalence of norms we have (see \cite[Theorem~2.1]{BenCasteMeht2008})
\begin{equation*}
\|u\|^{2}_{B_{s}}=\|u\|^{2}_{H^{s}}+\||y|^{s}u\|^{2}_{L^{2}}.
\end{equation*}
For $\gamma\in \mathbb{Z}$, we set $I_{\gamma}=\pi[\gamma-1, \gamma+1)$. Let $\ell^{p}_{\gamma}L_{t}^{q}(I_{\gamma}; L_{x}^{r}(\mathbb{R}^{d}))$ be the space of measurable functions $u:\mathbb{R}\rightarrow L_{x}^{r}(\mathbb{R}^{d})$ such that the norm $\|u\|_{\ell^{p}_{\gamma}L^{q}(I_{\gamma}; L_{x}^{r}(\mathbb{R}^{d}))}$ is finite, with
\begin{equation*}
\|u\|^{p}_{\ell^{p}_{\gamma}L_{t}^{q}(I_{\gamma}; L_{x}^{r}(\mathbb{R}^{d}))}=\sum_{\gamma\in \mathbb{Z}}\|u\|^{p}_{L_{t}^{q}(I_{\gamma}; L_{x}^{r}(\mathbb{R}^{d}))}.
\end{equation*}
To simplify the notation, we will use $\|u\|_{\ell^{p}_{\gamma}L^{q} L^{r}}$ when it is not ambiguous. Finally, we write $\|u\|_{\ell^{p}_{\gamma_{0}\leq\gamma\leq\gamma_{1}}L^{q}(I_{\gamma}; L_{x}^{r})}$ to signify
\begin{equation*}
\|u\|^{p}_{\ell^{p}_{\gamma_{0}\leq\gamma\leq\gamma_{1}}L^{q}(I_{\gamma}; L_{x}^{r})}=\sum_{\gamma_{0}\leq\gamma\leq\gamma_{1}}\|u\|^{p}_{L_{t}^{q}(I_{\gamma}; L_{x}^{r}(\mathbb{R}^{d}))}.
\end{equation*}

\section{Strichartz estimates and scattering}\label{S:es} 

\subsection{Local Strichartz estimates and local well-posedness}

Denote the (partial) harmonic potential by $V(x) = |y|^2$ (recall that $x=(y,z)\in
\R^n\times \R^{d-n}$). As $V$ is quadratic, it enters the
general framework of at most quadratic smooth potentials considered in
\cite{Fujiwara}. In particular, the propagator associated to
$H$ enjoys local dispersive estimates (as can be seen also from
generalized Mehler formula, see e.g. \cite{HormanderQuad})
\begin{equation*}
  \|e^{-itH}\|_{L^1(\R^d)\to L^\infty(\R^d)}\lesssim
  \frac{1}{|t|^{d/2}},\quad |t| \le 1,
\end{equation*}
which in turn imply \emph{local in time} Strichartz estimates,
\begin{equation*}
  \|e^{-itH}u_0\|_{L^q(I;L^r(\R^d))}\le C_q(I)
  \|u_0\|_{L^2(\R^d)},\quad
  \frac{2}{q}=d\(\frac{1}{2}-\frac{1}{r}\),\quad 2\le r<\tfrac{2d}{d-2},
\end{equation*}
where the constant $C_q(I)$ actually depends on $|I|$. Indeed, we
compute for instance
\begin{equation*}
  e^{-itH}\(e^{-|y|^2/2}v_0(z)\) =
  e^{-|y|^2/2+in t}\(e^{it\Delta_{\R^{d-n}}}v_0\)(z). 
\end{equation*}
Local in time Strichartz estimates suffice to establish local
well-posedness in the energy space, as proved in \cite{Carles2011}. We
give some elements of proof which introduce some useful vector fields. 
\begin{lemma}\label{lem:LWP}
  Let $d\ge 2$, $1\le n\le d-1$, $0<\si<\tfrac{2}{d-2}$, and
  $u_0\in B_1$. There exists $T=T(\|u_0\|_{B_1})$ and a unique
  solution $u\in C([-T,T];B_1)\cap
  L^{\frac{4\si+4}{d\si}}([-T,T];L^{2\si+2}(\R^d))$ to \eqref{GP}. In
  addition, the 
  conservations \eqref{Mome} hold. \\
  Either the solution is global in positive time, $u\in  C(\R_+;B_1)\cap
  L^{\frac{4\si+4}{d\si}}_{\rm loc}(\R_+;L^{2\si+2}(\R^d))$, or there
  exists $T_+>0$ such that
  \begin{equation*}
    \|\nabla_x u(t)\|_{L^2(\R^d)}\Tend t {T_+} \infty.
  \end{equation*}
  If $\lambda=+1$, then the
  solution is global in time, $u\in C(\R;B_1)\cap
  L^{\frac{4\si+4}{d\si}}_{\rm loc}(\R;L^{2\si+2}(\R^d))$.
\end{lemma}
\begin{proof}[Sketch of the proof]
  The proof relies on a classical fixed point argument applied to
  Duhamel's formula
  \begin{equation*}
    u(t) = e^{-itH}u_0 -i\lambda\int_0^t
    e^{-i(t-s)H}\(|u|^{2\si}u\)(s)ds, 
  \end{equation*}
  using (local in time) Strichartz estimates. The gradient $\nabla_z$
  commutes with $ e^{-itH}$, since there is no potential in the
  $z$ variable. On the other hand, in the $y$ variable, the presence
  of the harmonic potential ruins this commutation property. It is
  recovered by considering the vector fields
  \begin{equation*}
    A_1(t) = y\sin(2 t) -i\cos(2 t)\nabla_y,\quad A_2(t) =-
    y\cos(2 t)-i\sin(2 t)\nabla_y . 
  \end{equation*}
  We recall from e.g. \cite[Lemma~4.1]{AntoCarSilva2015} the main
  properties that we will use:
  \begin{equation*}
  \begin{pmatrix}
    A_1(t)\\
A_2(t)
  \end{pmatrix}
= 
\begin{pmatrix}
  \sin (2  t) & \cos (2  t)\\
-\cos (2 t) & \sin (2 t)
\end{pmatrix}
\begin{pmatrix}
  y \\
-i\nabla_y
\end{pmatrix},
\end{equation*}
they correspond to the conjugation of gradient and momentum by
   the free flow,
   \begin{equation*}
     A_1 (t) = e^{-itH }(-i\nabla_y)e^{itH },\quad A_2(t) =
    - e^{-itH}y\, e^{itH} ,
   \end{equation*}
and therefore, they commute with the linear part of \eqref{GP}:
   $[i\d_t-H,A_j(t)]=0$.
These vector fields  act on gauge invariant nonlinearities like
derivatives, and we have the pointwise estimate
  \begin{equation*}
    \left| A_j(t)\(|u|^{2\si}u\) \right|\lesssim
    |u|^{2\si}|A_j(t)u|.
  \end{equation*}
Once all of this is noticed, we can just mimic the standard proof of
local well-posedness of NLS in $H^1(\R^d)$ 
(see e.g. \cite{CB}), by considering $(A_1(t),A_2(t),\nabla _z)$
instead of $(\nabla_y,\nabla_z)$. The conservations \eqref{Mome}
follow from classical arguments (see e.g. \cite{CB}).
\smallbreak

From the construction, either the solution is global, or the
$B_1$-norm becomes unbounded in finite time. Like in the statement of
the lemma, we consider positive time only, the case of negative time
being similar. The obstruction to global existence reads
\begin{equation*}
  \|u(t)\|_{B_1}\Tend t {T_+}\infty,
\end{equation*}
for some $T_+>0$. But a standard virial computation yields
\begin{equation*}
  \frac{d}{dt}\|y u(t)\|_{L^2}^2 = 4\IM \int_{\R^d} \bar u(t,x)y\cdot
  \nabla_y u(t,x)dy.
\end{equation*}
Cauchy-Schwarz inequality shows that if $\|\nabla_y u(t)\|_{L^2}$
remains bounded locally in time, then so does $\|y u(t)\|_{L^2}$,
hence the blow-up criterion. Global
existence in the case $\lambda=+1$ is straightforward.
\end{proof}
For future reference, we note that
\begin{equation}\label{eq:vector-norm}
  \|e^{itH}u(t)\|_{B_1}^2 \sim \sum_{A\in \{{\rm Id},A_1,A_2,\nabla_z\}}
  \|A(t)u(t)\|_{L^2(\R^d)}^2. 
\end{equation}
\subsection{Global Strichartz estimates}

To prove scattering results, we use global in time Strichartz
estimates, taking advantage of the full dispersion in the $z$
variable, and of the local dispersion in the total variable $x=(y,z)$.
\begin{lemma}[Global Strichartz estimates, Theorem~3.4 from \cite{AntoCarSilva2015}]\label{Ls1} 
Let $d\geq2$, $1\leq n\leq d-1$ and $2\leq r<\tfrac{2d}{d-2}$. Then the solution $u$ to $(i\partial_{t}-H)u=F$ with initial data $u_{0}$ obeys
\begin{equation}\label{Sz}
\|u\|_{\ell^{p_{1}}_{\gamma}L^{q_{1}}L^{r_{1}}}\lesssim
\|u_{0}\|_{L^{2}(\mathbb{R}^{d})} +\|F\|_{\ell_{\gamma}^{{p}_{2}^{\prime}}L^{{q}_{2}^{\prime}}L^{{r}_{2}^{\prime}}} ,
\end{equation}
provided that the following conditions hold:
\begin{equation}\label{paa1}
\frac{2}{q_{k}}=d\left(\frac{1}{2}-\frac{1}{r_{k}}\right),\quad\frac{2}{p_{k}}=(d-n)\left(\frac{1}{2}-\frac{1}{r_{k}}\right), \quad k=1, 2.\\
\end{equation}
\end{lemma}

Moreover, as in e.g. \cite{HolmerRoudenko2008} or \cite{TzVi2016},  we
will need the following inhomogeneous Strichartz estimates.

\begin{lemma}[Inhomogeneous Strichartz estimates]\label{Ls22}
Let $d\geq2$, $1\leq n\leq d-1$. Then  we have 
\begin{equation*}
   \left\| \int_0^t e^{-i(t-s)H} u(s)ds\right\|_{\ell^{p}_{\gamma}L^{q}L^{r}}\lesssim
   \|u\|_{\ell^{\tilde{p}'}_{\gamma}L^{\tilde{q}'}L^{{r}'}}, 
 \end{equation*}
provided that $q,\tilde q\in [1,\infty]$ and:
\begin{gather*}
   \frac{2}{p}+\frac{2}{\tilde p} = (d-n)\left(1-\frac{2}{r}\right),\\
   \frac{1}{p}+\frac{d-n}{r}<\frac{d-n}{2},\quad
                             \frac{1}{\tilde p}+\frac{d-n}{r}<\frac{d-n}{2}, \quad
                              \text{(acceptable pairs)}\\
   \frac{1}{p}+\frac{1}{\tilde p}<1.
 \end{gather*}
\end{lemma}
\begin{proof}
The proof of the inhomogeneous Strichartz estimates for non-admissible
pairs is a  direct adaptation of the proof of Theorem~1.4  in
\cite{Foschi2005}. We emphasize that we consider
the same Lebesgue index in space on the left and right hand sides in
the above inequality, which makes the adaptation of \cite[Theorem
1.4]{Foschi2005}  easier.
\end{proof}
We will also need a weaker dispersive property:
\begin{lemma}\label{lem:dispLp}
  Let $1\le n\le d-1$ and $2<r<\tfrac{2d}{d-2}$. For any
$\varphi\in B_1$,
\begin{equation*}
\|e^{-itH}\varphi\|_{L^r(\R^d)}\Tend t {\pm \infty} 0.      
\end{equation*}
\end{lemma}
This result is actually valid more generally if the harmonic potential
$|y|^2$ is replaced by a
potential bounded from below, as shown by the proof. 
\begin{proof}
  When $\varphi$ belongs to the conformal space, $\varphi\in \Sigma$,
  we consider the Galilean operator in $z$ (see e.g. \cite{GV79Scatt,CB}),
  \begin{equation*}
    J_z(t) = z+2it\nabla_z= 2it \, e^{i|z|^2/(4t)}\nabla_z\(\cdot \,  e^{-i|z|^2/(4t)}\).
  \end{equation*}
  Gagliardo-Nirenberg inequality yields
  \begin{equation*}
    \|e^{-itH}\varphi\|_{L^r(\R^d)}\lesssim
    |t|^{-\delta}
    \|e^{-itH}\varphi\|_{L^2(\R^d)}^{1-\delta}\|(\nabla_y,J_z(t))e^{-itH}\varphi\|_{L^2(\R^d)}^\delta,
  \end{equation*}
  where $   \delta = (d-n)\(\tfrac{1}{2}-\tfrac{1}{r}\)$. Since the
  harmonic potential is non-negative,
  \begin{equation*}
    \|(\nabla_y,J_z(t))e^{-itH}\varphi\|_{L^2(\R^d)}\lesssim
    \|(\(-\Delta_y+|y|^2\)^{1/2},J_z(t))e^{-itH}\varphi\|_{L^2(\R^d)}, 
  \end{equation*}
  and since the operator $\(-\Delta_y+|y|^2\)^{1/2}$ commutes
  with $e^{-itH}$, which is unitary on $L^2(\R^d)$, and
  \[ J_{z}(t) = e^{it\Delta_z} z e^{-it\Delta_z}= e^{-itH} z e^{itH},\]
  we infer
  \begin{equation*}
     \|e^{-itH}\varphi\|_{L^r(\R^d)}\lesssim
    |t|^{-\delta}
    \|\varphi\|_{\Sigma}.
  \end{equation*}
  In view of Sobolev embedding and the fact that $e^{-itH}$ preserves
  the $B_1$-norm,
  \begin{equation*}
    \|e^{-itH}\varphi\|_{L^r(\R^d)}\lesssim
    \|e^{-itH}\varphi\|_{H^1(\R^d)}\lesssim
    \|e^{-itH}\varphi\|_{B_1}=\|\varphi\|_{B_1}, 
  \end{equation*}
  the result follows by a density argument.
\end{proof}
\subsection{Fixing Lebesgue indices for the scattering analysis}
\label{sec:indices}

From now on, we fix the exponents $\tilde{{q}}$, $\tilde{{p}}$, $p$,
$q$, $p_{0}$, $q_{0}$, $r$ as follows.

\begin{lemma}\label{lem:indices}
  Let $\tfrac{2}{d-n}\le \si<\frac{2}{d-2}$, and set
  \begin{align*}
&\tilde{{q}}=\frac{4\sigma(\sigma+1)}{2d\sigma^{2}+\sigma(d-2)-2},\quad
                 \tilde{{p}}=\frac{4\sigma(\sigma+1)}{2d\sigma^{2}+\sigma(d-2-n)-2(n\sigma^{2}+1)},\\ 
& p=\frac{4\sigma(\sigma+1)}{2\sigma+2-(d-n)\sigma},\quad q=\frac{4\sigma(\sigma+1)}{2\sigma+2-d\sigma},\quad r=2\sigma+2,\\
& p_{0}=\frac{4\sigma+4}{(d-n)\sigma},\quad q_{0}=\frac{4\sigma+4}{d\sigma}.
\end{align*}
Then the triplet $(p_{0}, q_{0}, r)$ satisfies the condition
\eqref{paa1}. Moreover, the triplets $(p, q, r)$ and $(\tilde{p},
\tilde{q}, r)$ satisfy the conditions in Lemma~\ref{Ls22}. 
\end{lemma}
\begin{proof}
  That the triplet $(p_{0}, q_{0}, r)$ satisfies the condition
  \eqref{paa1} is readily checked.
  \smallbreak
  We note that $\tilde q\in [1,\infty]$ iff $\tilde q'\in [1,\infty]$. Thus we
must check that $q\ge 2\sigma+1$. In turn this inequality follows
provided that $4\sigma(\sigma+1)\ge (2\sigma+1)(2\sigma+2-d\sigma)$
and it is equivalent to $\sigma\ge
\sigma_c(d)={2-d+\sqrt{d^2-12d+4}}/{4d}$, a threshold which is
classical in scattering theory for NLS (see e.g. \cite{CB}). Since $\sigma_c(d)<2/d<2/(d-n)$, the condition is fulfilled. Now we focus on the exponent $\tilde p$. We compute
 \begin{equation*}
   \frac{1}{\tilde p} = (d-n)\frac{2\sigma+1}{4\sigma+4}-\frac{1}{2\sigma},
 \end{equation*}
 and thus
 \begin{equation*}
   \frac{2}{p}+\frac{2}{\tilde p} =
   (d-n)\frac{2\sigma}{2\sigma+2}=(d-n)\left(1-\frac{1}{r}-\frac{1}{r}\right). 
 \end{equation*}
We also have, from the above formula,
 \begin{equation*}
    \frac{1}{p}+\frac{1}{\tilde p} =
    (d-n)\frac{\sigma}{2\sigma+2}<1,\quad\text{since}\quad
    \sigma<\frac{2}{d-2}<\frac{2}{(d-n-2)_+}. 
  \end{equation*}
  All that remains is to check that we have acceptable pairs:
\begin{equation*}
    \frac{1}{p}+\frac{d-n}{r}<\frac{d-n}{2}\Longleftrightarrow
 \frac{1}{2\sigma}+\frac{d-n}{4\sigma+4}<\frac{d-n}{2}. 
\end{equation*}
Since $\sigma\geq 2/(d-n)$, we infer that
\begin{equation*}
\frac{1}{2\sigma}+\frac{d-n}{4\sigma+4}\leq \frac{d-n}{4}+\frac{d-n}{4\sigma+4},
\end{equation*}
and the above inequality is satisfied as soon as
\begin{equation*}
\frac{d-n}{4\sigma+4}< \frac{d-n}{4},
\end{equation*}
which is trivially the case.
Last, we check
\begin{equation*}
  \frac{1}{\tilde p}+\frac{d-n}{r}<\frac{d-n}{2}\Longleftrightarrow
  \frac{d-n}{4\sigma+4}<\frac{1}{2\sigma},
\end{equation*}
which is again the case since
\begin{equation*}
   \sigma<\frac{2}{d-2}<\frac{2}{(d-n-2)_+}. 
 \end{equation*}
\end{proof}

We note that $(q_0,r)$ corresponds to the admissible pair
appearing in Lemma~\ref{lem:LWP}.

\subsection{Scattering}
\label{sec:scattering1}
The interest of the specific choice for $(p,q,r)$ appears in the
following lemma.

\begin{lemma}\label{Ls2}
Let $u_{0}\in B_{1}$ and $u$ be the corresponding solution of Cauchy
problem  \eqref{GP} with $u(0)=u_{0}$. If $u$ is global, $u\in C(\R;B_1)\cap
  L^{\frac{4\si+4}{d\si}}_{\rm loc}(\R;L^{2\si+2}(\R^d))$, and satisfies
\[  \| u\|_{\ell_{\gamma}^{p}L^{q}L^{r}}<\infty,\]
then the solution $u$ scatters in ${B}_{1}$ as $t\rightarrow\pm\infty$. 
\end{lemma}
\begin{proof}
We first show that $\|
A u\|_{\ell_{\gamma}^{p_{0}}L^{q_{0}}L^{r}}<\infty$ for
all $A\in \{ {\rm Id},A_1,A_2,\nabla_z\}$. As
$A u\in L^{q_0}_{\rm loc}(\R;L^r)$, we need to show that for
$\gamma_0\gg 1$, $\|
A u\|_{\ell_{\gamma\ge
    \gamma_0}^{p_{0}}L^{q_{0}}(I_{\gamma},L^{r})}<\infty$, the case of
negative times being similar.
We consider the integral equation
\[
  u(t)=e^{-i(t-\pi\gamma_{0})H}u(\pi\gamma_{0})-i\lambda\int_{\pi\gamma_{0}}^{t}e^{-i(t-s)H}(|u|^{2\sigma}u)(s)ds.\]
Notice the algebraic identities,
\begin{equation*}
  \frac{1}{p_0'}= \frac{1}{p_0}+\frac{2\si}{p},\quad
                  \frac{1}{q_0'} = \frac{1}{q_0}+\frac{2\si}{q}. 
 \end{equation*}
For $\gamma_1>\gamma_0>0$, 
  Strichartz estimate (Lemma~\ref{Ls1}) and H\"older inequality yield
\begin{align*}
   \|A u\|_{\ell_{\gamma_0\le \gamma\le
   \gamma_1}^{p_0}L^{q_0}L^r}&\lesssim
   \|A e^{-itH}u_0\|_{\ell_{\gamma_0\le \gamma\le
   \gamma_1}^{p_0}L^{q_0}L^r}+ \||u|^{2\si}Au\|_{\ell_{\gamma_0\le \gamma\le
   \gamma_1}^{p_0'}L^{q_0'}L^{r'}}\\
   &\lesssim
   \|A u_0\|_{L^2}+ \|  u\|^{2\si}_{\ell_{\gamma_0\le \gamma\le
   \gamma_1}^{p}L^{q}L^{r}}\|  A u\|_{\ell_{\gamma_0\le \gamma\le
   \gamma_1}^{p_0}L^{q_0}L^{r}}.
 \end{align*}
For  $\gamma_0\gg 1$ so that $\|
u\|_{\ell^{p}_{\gamma\ge \gamma_0}L^{q}L^{r}}$ is sufficiently small,
a bootstrap argument yields
\begin{equation*}
  \|A u\|_{\ell^{p_0}_{\gamma_0\le \gamma\le
   \gamma_1}L^{q_0}L^r}\lesssim
   \|A u_0\|_{L^2}\lesssim \|u_0\|_{B_1},
 \end{equation*}
 uniformly in $\gamma_1>\gamma_0$, hence $Au\in \ell^{p_0}_\gamma
 L^{q_0}L^r$.
 \smallbreak

 Using Strichartz estimates again, we have, for $t_2>t_1>0$, 
\begin{align*} 
\|A (t_2)u(t_2)-A(t_1)u(t_1)\|_{L^2}&= \left\|
  \int^{t_{2}}_{t_{1}}e^{isH}A(s)(|u|^{2\sigma}u)(s)ds\right\|_{L^2}\\
  &
  \lesssim \left\|A\(|u|^{2\si}u\)\right\|_{\ell^{p_0'}_{\gamma\gtrsim
    t_1}L^{q_0'}L^{r'}}\\
  &\lesssim \|u \|^{2\si}_{\ell^{p}_{\gamma\gtrsim
    t_1}L^{q}L^{r}}\|Au\|_{\ell^{p_0}_{\gamma\gtrsim
    t_1}L^{q_0}L^{r}}\Tend {t_1}\infty 0,
\end{align*}
and so, in view of \eqref{eq:vector-norm}, $e^{itH}u(t)$ converges strongly in $B_1$ as $t\to \infty$. 
\end{proof}

With Duhamel's formula in mind, we show that the homogeneous part
always belong to the scattering space considered in Lemma~\ref{Ls2}.

\begin{lemma}\label{axc}  
Let $\psi\in B_{1}$. Then
\begin{equation}\label{Imi}
\|e^{-itH}\psi\|_{\ell_{\gamma}^{{p}}L^{q}L^{r}} \lesssim  \|\psi\|_{B_{1}}.
\end{equation}
\end{lemma}
\begin{proof}
  We recall some details of the proof of \cite[Theorem
  3.4]{AntoCarSilva2015}. 
Consider a partition of unity 
\[ \sum_{\gamma \in \mathbb{Z}} \chi(t-\pi \gamma)=1, \quad \forall {t \in \mathbb{R}}
   \qquad \mbox{with} \quad \mbox{supp}\chi \subset [-\pi, \pi].\]
 Lemma~\ref{Ls1} is actually proven by considering
 \begin{equation*}
   \|\psi\|^{p}_{\ell_{\gamma}^{p}L^{q}L^{r}} =\sum_{\gamma\in \mathbb{Z}} \|\chi(\cdot
   -\gamma\pi) \psi\|_{L^q(\mathbb{R};L^r(\mathbb{R}^d))}^p. 
 \end{equation*}
 By Sobolev embedding,
 \begin{equation*}
   \|\chi(\cdot
   -\gamma\pi) e^{-itH}\psi\|_{L^q(\mathbb{R};L^r(\mathbb{R}^d))}\lesssim \|\chi(\cdot
   -\gamma\pi) e^{-itH}\psi\|_{W^{s,k}(\mathbb{R};L^r(\mathbb{R}^d))} ,\quad \frac{1}{q}=\frac{1}{k}-s.
 \end{equation*}
We note the  relations
 \begin{equation*}
   \frac{2}{p_0} = (d-n)\(\frac{1}{2}-\frac{1}{r}\), \quad
   \frac{2}{p} = (d-n)\(\frac{1}{2}-\frac{1}{r}\)-
   \(\frac{d-n}{2}-\frac{1}{\si}\),
 \end{equation*}
hence  $p\ge p_0$ since $\si\ge \tfrac{2}{d-n}$. Therefore,
 \begin{equation}\label{auxn1}
   \| e^{-itH}\psi\|_{\ell_{\gamma}^pL^q L^r}\lesssim \|\chi(\cdot
   -\gamma\pi)e^{-itH}\psi\|_{\ell^{p_0}_\gamma W^{s,k}(\mathbb{R};L^r(\mathbb{R}^d))} .
 \end{equation}
 If we set $k=q_0$ (in order to recover our initial 
 triplet), we find
 \begin{equation*}
   \frac{1}{q}= \frac{1}{2\sigma}-\frac{d}{4\sigma+4} =
   \underbrace{\frac{d}{2}\left(\frac{1}{2} -\frac{1}{2\sigma+2}\right)}_{=1/q_0} -s,\quad
   \text{hence}\quad s:= \frac{1}{2}{\left(\frac{d}{2}-\frac{1}{\sigma}\right)}.
 \end{equation*}
 Using
 \begin{align*}
   \|\chi(\cdot
   -\gamma\pi) e^{-itH}\psi\|_{W^{s,q_0}(\mathbb{R};L^r(\mathbb{R}^d))}&\lesssim  \|H^s\chi(\cdot
   -\gamma\pi) e^{-itH}\psi\|_{L^{q_0}(\mathbb{R};L^r(\mathbb{R}^d))}\\
   &\lesssim \|\chi(\cdot
   -\gamma\pi) e^{-itH}H^s \psi\|_{L^{q_0}(\mathbb{R};L^r(\mathbb{R}^d))},
 \end{align*}
 the homogeneous Strichartz estimate yields
 \begin{equation*}
    \| e^{-itH}\psi\|_{\ell_{\gamma}^pL^q L^r}\lesssim \|\chi(\cdot
   -\gamma\pi) e^{-itH}H^{s} \psi\|_{\ell^{p_0}_\gamma
     L^{q_0}(\mathbb{R};L^r(\mathbb{R}^d))}\lesssim
   \|\psi\|_{B_{2s}}\lesssim \|\psi\|_{B_{1}},
 \end{equation*}
since $0<s<\tfrac{1}{2}$, as $\tfrac{2}{d}<\tfrac{2}{d-n}\le \si<\tfrac{2}{d-2}$.
\end{proof}

\section{Variational estimates}\label{S:1} 
From now on, we assume $\lambda=-1$.
\smallbreak

We define on $B_{1}$ the Nehari functional
\begin{equation*}
I(u)=\|\nabla_{x} u\|^{2}_{L^{2}}+\|yu\|^{2}_{L^{2}}+\|u\|^{2}_{L^{2}}-\|u\|^{2\sigma+2}_{L^{\sigma+2}}.
\end{equation*}
In this section we show that the set of ground states is not
empty. Moreover, we prove that  $I(u)$ and $P(u)$ have the
same sign under the condition $S(Q)<S(u)$, which
plays a vital role in the proof of Theorem \ref{Th1}.  Here $Q$ is a
ground state. 
To prove this, we introduce the scaling quantity $\varphi^{a,b}_{\lambda}$ by
\begin{equation}\label{Fm}
\varphi^{a,b}_{\lambda}(x)=e^{a \lambda}\varphi(y, e^{-b\lambda}z), \quad  x=(y,z)\in\mathbb{R}^{n}\times  \mathbb{R}^{d-n},
\end{equation}
where $(a,b)$ satisfies the following conditions
\begin{equation}\label{Cd1}
a>0, \quad b\leq 0, \quad 2a+b(d-n)\geq 0, \quad \sigma a+b> 0, \quad (a,b)\neq (0,0).
\end{equation}
A simple calculation shows that
\begin{align*}
&\|\nabla_{y} \varphi^{a,b}_{\lambda}\|^{2}_{L^{2}}=e^{\lambda(2a+b(d-n))}\|\nabla_{y} \varphi\|^{2}_{L^{2}},\quad\|\nabla_{z} \varphi^{a,b}_{\lambda}\|^{2}_{L^{2}}=e^{\lambda(2a+b(d-n-2))}\|\nabla_{z} \varphi\|^{2}_{L^{2}},
\\
&\| \varphi^{a,b}_{\lambda}\|^{2}_{L^{2}}=e^{\lambda(2a+b(d-n))}\|\varphi\|^{2}_{L^{2}},\quad\|\varphi^{a,b}_{\lambda}\|^{2\sigma+2}_{L^{2\sigma+2}}=e^{\lambda(a(2\sigma+2)+b(d-n))}\| \varphi\|^{2\sigma+2}_{L^{2\sigma+2}},\\
&\|y \varphi^{a,b}_{\lambda}\|^{2}_{L^{2}}=e^{\lambda(2a+b(d-n))}\|y\varphi\|^{2}_{L^{2}}.
\end{align*}
We define the functionals $J^{a,b}$ by
\begin{align*}
J^{a,b}(\varphi)&=\left.\partial_{\lambda}S(\varphi^{a,b}_{\lambda})\right|_{\lambda=0}\\
&=\frac{2a+b(d-n)}{2}\|\nabla _{y}
                                                                                                              \varphi\|^{2}_{L^{2}}+\frac{2a+b(d-n-2)}{2}\|\nabla _{z} \varphi\|^{2}_{L^{2}}\\
  &\quad+\frac{2a+b(d-n)}{2}\|y \varphi\|^{2}_{L^{2}}\\
&\quad+\frac{2a+b(d-n)}{2}\|\varphi\|^{2}_{L^{2}}-\frac{a(2\sigma+2)+b(d-n)}{2\sigma+2}\|\varphi\|^{2\sigma+2}_{L^{2\sigma+2}}.
\end{align*}
In particular, when $(a,b)=(1,0)$ and $(a,b)=(1,-2/(d-n))$ we obtain the functionals $I$ and $P$ respectively.
In the next result, we see that $J^{a,b}$ is positive near
the origin in the space ${B}_{1}$.
\smallbreak

As a technical preliminary, denote
 \[\|u\|_{\dot{B}_{1}}^2=\|\nabla_{x}u\|^{2}_{L^{2}}+\|yu\|^{2}_{L^{2}}\]
 the homogeneous counterpart of the $B_1$-norm. From the uncertainty
 principle in $y$, and 
Cauchy-Schwarz inequality in $z$,
\begin{equation*}
  \|\varphi\|_{L^2}^2 \le \frac{2}{n}\|y\varphi\|_{L^2}\|\nabla_y \varphi\|_{L^2}.
\end{equation*}
In particular, $\|u\|_{B_1}\sim \|u\|_{\dot B_1}$.
 
\begin{lemma} \label{LP1}
  Let $(a,b)$ satisfying \eqref{Cd1}, with in addition $2a+b(d-n)>0$.
Let $\left\{v_{k}\right\}^{\infty}_{k=1}\subset {B}_{1}  \setminus
\left\{0 \right\}$ be bounded in ${B}_{1}$ such that
$\lim_{k\rightarrow \infty}\|v_{k}\|_{\dot{B}_{1}}=0$. Then
for sufficiently large $k$, we have $J^{a,b}(v_{k})>0$.  
\end{lemma}
\begin{proof}
 Gagliardo-Nirenberg inequality yields
\begin{align*}
J^{a,b}(v_{k})&\geq \frac{2a+b(d-n)}{2}\|v_{k}\|_{\dot{B}_{1}}^2- \frac{a(2\sigma+2)+b(d-n)}{2\sigma+2}\|v_k\|^{2\sigma+2}_{L^{2\sigma+2}}  \\
&\geq\frac{2a+b(d-n)}{2}\|v_{k}\|_{\dot{B}_{1}}^2- \frac{a(2\sigma+2)+b(d-n)}{2\sigma+2}C\|v_{k}\|^{2\sigma +2}_{\dot{B}_{1}},
\end{align*}
where $C$ is a positive constant. 
Since $2a+b(d-n)>0$,  we infer that for sufficiently large $k$, $J^{a,b}(v_{k})>0$. This proves the lemma.
\end{proof}

Next, we consider the minimization problem
\begin{align}\label{MPE2}
d^{a,b}&:={\inf}\left\{S(u):\, u\in  {B}_{1}  \setminus  \left\{0 \right\},  J^{a,b}(u)=0\right\},\\
{U}^{a,b}&=\bigl\{ \varphi\in {B}_{1}: S(\varphi)=d^{a,b}\quad \text{and}\quad J^{a,b}(u)=0\bigl\}.
\end{align}
\begin{lemma} \label{L1}
Let $(a,b)$ satisfying \eqref{Cd1}, with in addition
$2a+b(d-n)>0$. Then the set  ${U}^{a,b}$ is not empty. That
is, there exists $Q\in {B}_{1}$ such that  
$S(Q)=d^{a,b}$ and $J^{a,b}(Q)=0$.
\end{lemma}
\begin{proof}
We introduce the functional
\begin{align}\label{Co1}
B^{a,b}(u)&=S(u)-\frac{1}{a(2\sigma+2)+b(d-n)}J^{a,b}(u)\\\nonumber
&=\alpha_{1}\|\nabla _{y} u\|^{2}_{L^{2}}+\alpha_{2}\|\nabla _{z} u\|^{2}_{L^{2}}+\alpha_{1}\|y u\|^{2}_{L^{2}}+\alpha_{1}\|u\|^{2}_{L^{2}},
\end{align}
where 
\begin{equation*}
\alpha_{1}:=\frac{1}{2}\left(1-\frac{2a+b(d-n)}{a(2\sigma+2)+b(d-n)}\right)>0, \quad \alpha_{2}:=\frac{1}{2}\left(1-\frac{2a+b(d-n-2)}{a(2\sigma+2)+b(d-n)}\right)>0.
\end{equation*}
To claim that $\alpha_2>0$, we have used $\sigma a+b>0$. From \eqref{Co1}, it is clear that there exist constants $C_{1}$, $C_{2}>0$ such that for all $u\in {B}_{1}$,
\begin{equation}\label{Eqi2}
C_{1}\| u\|^{2}_{{B}_{1}}\leq B^{a,b}(u) \leq C_{2}\| u\|^{2}_{{B}_{1}}.
\end{equation}
Notice that
\begin{equation}\label{VP21}
d^{a,b}={\inf}\left\{B^{a,b}(u):\, u\in {B}_{1} \setminus  \left\{0 \right\},  J^{a,b}(u)=0\right\}.
\end{equation}
\textbf{Step 1.} We claim that $d^{a,b}>0$.
Indeed, let $u\neq0$ such that $J^{a,b}(u)=0$.
Then we have, in view of \eqref{Cd1} and since $2a+b(d-n)>0$,
\begin{equation*}
  \|u\|_{B_1}^2\lesssim \|u\|^{2\sigma+2}_{L^{2\sigma+2}}\lesssim
  \|u\|_{B_1}^{2\si+2}, 
\end{equation*}
where we have used Gagliardo-Nirenberg inequality and the uncertainty
principle like in the previous proof. This implies $\|u\|_{B_1}\gtrsim
1$, hence $B^{a,b}(u)\gtrsim 1$ from  \eqref{Eqi2}.\\
\textbf{Step 2.} If $u\in  {B}_{1}$ satisfies $J^{a,b}(u)<0$, then $d^{a,b}<B^{a,b}(u)$. Indeed,
as $J^{a,b}(u)<0$, a simple calculation shows that there exists $\lambda\in (0,1)$ such that $J^{a,b}(\lambda u)=0$. Thus, by definition of $d^{a,b}$, we obtain
\begin{equation*}
d^{a,b}\leq B^{a,b}(\lambda u)=\lambda^{2} B^{a,b}(u)<B^{a,b}(u).
\end{equation*}
\textbf{Step 3.} We will need the following result that was proved in \cite[Lemma 3.4]{BEBOJEVI2017} (see also \cite{OhtaPartial2018}): assume that the sequence $\left\{u_{k}\right\}^{\infty}_{k=1}$ is bounded in ${B}_{1}$ and satisfies 
\begin{equation*}
\limsup_{k\rightarrow\infty} \|u_{k}\|^{2\sigma +2}_{L^{2\sigma +2}}\geq C>0.
\end{equation*}
 Then, there exist a sequence $\left\{z_{k}\right\}^{\infty}_{k=1}\subset \mathbb{R}^{d-n}$ and $u\neq 0$ such that, passing to a subsequence if necessary
\begin{equation*}
\tau_{z_{k}}u_{k}(y,z):=u_{k}(y,z-z_{k})\rightharpoonup u \quad \text{weakly in ${B}_{1}$.}
\end{equation*}
\textbf{Step 4.} We claim that ${U}^{a,b}$ is not empty. Let $\left\{u_{k}\right\}^{\infty}_{k=1}$ be a minimizing sequence of $d^{a,b}$. Since $B^{a,b}(u_{k})\rightarrow d^{a,b}$ as $k$ goes to $\infty$, by \eqref{Eqi2} we infer that the sequence 
$\left\{u_{k}\right\}^{\infty}_{k=1}$ is bounded in ${B}_{1}$. Moreover, as $J^{a,b}(u_{k})=0$ we have
\begin{align*}
\|u_{k}\|^{2\sigma +2}_{L^{2\sigma +2}}\gtrsim \|u_k\|_{B_1}^2\gtrsim
B^{a,b}(u_{k}) \rightarrow d^{a,b}>0,
\end{align*}
as $k\rightarrow \infty$. Therefore, $\limsup_{k\rightarrow\infty}
\|u_{k}\|^{2\sigma+2}_{L^{2\sigma +2}}\geq C>0$. Thus, by Step~3 there
exist a sequence $\left\{z_{k}\right\}\subset \mathbb{R}^{d-n}$ and
$u\neq 0$ such that $\tau_{z_{k}}u_{k}\rightharpoonup u$ weakly in
${B}_{1}$. We set $v_{k}(x):=\tau_{z_{k}}u_{k}(x).$ Now, we prove that
$J^{a,b}(u)=0$.  Suppose that $J^{a,b}(u)<0$. By the
weakly lower semicontinuity of $B^{a,b}$ and Step 2 we see
that 
\begin{equation*}
d^{a,b}< B^{a,b}(u)\leq \liminf_{k\rightarrow \infty}B^{a,b}(u_{k})=d^{a,b},
\end{equation*}
which is impossible. Now we assume that $J^{a,b}(u)>0$.  From Brezis-Lieb Lemma we get
\begin{equation*}
\lim_{n\rightarrow\infty }J^{a,b}(u_{n}-u)=\lim_{n\rightarrow\infty }\left\{J^{a,b}(u_{n})-J^{a,b}(u)\right\}=-J^{a,b}(u)<0.
\end{equation*}
This implies that $J^{a,b}(u_{n}-u)<0$ for sufficiently large $n$. Thus, applying the same argument as above, we see that
\begin{equation*}
d^{a,b}\leq  \lim_{n\rightarrow\infty } B^{a,b}(u_{n}-u)=\lim_{n\rightarrow\infty }\left\{ B^{a,b}(u_{n})- B^{a,b}(u)\right\}=d^{a,b}-B^{a,b}(u)<d^{a,b},
\end{equation*}
because $B^{a,b}(u)>0$. Therefore $J^{a,b}(u)=0$ and
\begin{equation*}
d^{a,b}\leq S(u)=B^{a,b}(u)\leq \liminf_{n\rightarrow\infty} B^{a,b}(u_{n})=d^{a,b}.
\end{equation*}
In particular, $S(u)=d^{a,b}$ and $u\in {U}^{a,b}$. This concludes the proof of lemma.
\end{proof}

\begin{remark}\label{Rgs}
Lemma~\ref{L1} shows that the set of ground states is not empty. Indeed, in the case $(a,b)=(1,0)$,  from Lemma \ref{L1} we have that 
there exists $Q\in B_{1}$ such that $S(Q)=\inf\left\{S(\varphi):\, I(\varphi)=0\right\}$. This implies  that (see \cite[Chapter 8]{CB})
\begin{equation*}
S(Q)=\inf\left\{S(\varphi):\, \text{ $\varphi$ is a solution of \eqref{Ep}}\right\}.
\end{equation*}
\end{remark}

Now we define the mountain pass level $\beta$ by setting
\begin{equation}\label{mpp}
\beta:=\inf_{\sigma\in\Gamma}\max_{s\in [0,1]}S(\sigma(s)),
\end{equation}
where $\Gamma$ is the set
\begin{align*}
{\Gamma}&:=\bigl\{\sigma\in C([0,1];{B}_{1}): \sigma(0)=0, S(\sigma(1))<0 \bigl\}.      
\end{align*}
\begin{lemma} \label{LL2} Let  $(a,b)$ satisfying \eqref{Cd1}, with in addition $2a+b(d-n)>0$.  We have the following properties.\\
(i)  The functional $S$ has a mountain pass geometry, that is $\Gamma\neq \emptyset$ and $\beta>0$.\\
(ii) The identity $\beta=d^{a,b}$ holds. In particular, if $Q$ is a ground state, then $S(Q)=\beta$.
\end{lemma}
\begin{proof}
(i) Let $v\in {B}_{1}\setminus\left\{0\right\}$. For $s>0$ we obtain
\begin{equation*}
S(sv)=s^{2}\|v\|^{2}_{B_{1}}-\frac{s^{2\sigma+2}}{2\sigma+2}\|v\|^{2\sigma+2}_{L^{2\sigma+2}}.
\end{equation*}
Let $L>0$ such that $S(Lv)<0$. We define $\sigma(s):=Lsv$. Then $\sigma\in C([0,1];{B}_{1})$, $\sigma(0)=0$ and $S(\sigma(1))<0$; this implies that $\Gamma$ is nonempty. On the other hand, notice that, by the embedding of ${B}_{1}\hookrightarrow L^{2\sigma+2}$ we have
\begin{equation*}
S(v)\geq \frac{1}{2}\|v\|^{2}_{{B}_{1}}-\frac{C}{2\sigma+2}\|v\|^{2\sigma+2}_{{B}_{1}}.
\end{equation*}
Taking $\eps>0$ small  enough we have
\begin{equation*}
\delta:=\frac{1}{2}\eps^{2}-\frac{C}{2\sigma+2}\eps^{2\sigma+2}>0.
\end{equation*}
Thus, if $\|v\|^{2}_{{B}_{1}}<\eps$, then $S(v)>0$. Therefore, for any $\sigma\in \Gamma$ we have $\|\sigma(1)\|^{2}_{{B}_{1}}>\eps$, and by continuity of $\sigma$, there exists $s_{0}\in [0,1]$ such that $\sigma(s_{0})=\eps$. This implies that
\begin{equation*}
\max_{s\in [0,1]}S(\sigma(s))\geq S(\sigma(s_{0}))\geq \delta>0.
\end{equation*}
By definition of $\beta$, we see that $\beta\geq \delta>0$.\\
(ii) Let $\sigma\in \Gamma$. Since $\sigma(0)=0$, by Lemma \ref{LP1} we infer that there exists $s_{0}>0$ such that $J^{a,b}(\sigma(s_{0}))>0$. Also we note that from \eqref{Co1} we have
\begin{align*}
J^{a,b}(\sigma(1))&=(a(2\sigma+2)+b(d-n))\left\{S(\sigma(1))-B^{a,b}(\sigma(1))\right\}\\
&<(a(2\sigma+2)+b(d-n))S(\sigma(1))<0.
\end{align*}
By continuity of $s\mapsto J^{a,b}(\sigma(s))$, we infer that there exists $s^{\ast}\in (0,1)$ such that $J^{a,b}(\sigma(s^{\ast}))=0$. This implies that
\begin{equation*}
\max_{s\in [0,1]}S(\sigma(s))\geq S(\sigma(s^{\ast}))\geq d^{a,b}.
\end{equation*}
Taking the infimum on $\Gamma$, we obtain $\beta \geq d^{a,b}$. Now we prove $\beta\leq d^{a,b}$. Let $\varphi\in {B}_{1}\setminus\left\{0\right\}$ be such that $J^{a,b}(\varphi)=0$. We put $f(s):=\varphi^{a,b}_{s}(y, z)$ for $s\in \mathbb{R}$, where $\varphi^{a,b}_{s}$ is defined in \eqref{Fm}. Notice that as $a\sigma+b>0$, it follows that $S(f(s))<0$ for sufficiently large $s>0$.
Since
$\left.\partial_{s}S(f(s))\right|_{s=0}=J^{a,b}(\varphi)=0$,
it follows that $\max_{s\in
  \mathbb{R}}S(f(s))=S(f(0))=S(\varphi)$. Let
$L>0$ be such that $S(f(L))<0$. We define 
\begin{equation*}
h(s):=
\begin{cases} 
f(s)\quad \text{if $-\frac{L}{2}\leq s \leq L$,}\\
\frac{2}{L}(s+L)f(-\frac{L}{2})\quad \text{if $-{L}\leq s \leq -\frac{L}{2}$.}
\end{cases} 
\end{equation*}
Then $s\mapsto h(s)$ is continuous in ${B}_{1}$, $S(h(L))<0$,
$S(h(-L))=0$ and
\[\max_{s\in
    [-L,L]}S(h(s))=S(h(0))=S(\varphi).\]
By changing variables, we infer that there exists $\sigma\in\Gamma$ such that $\max_{s\in [0,1]}S(\sigma(s))=S(\varphi)$. Thus,
\begin{equation*}
\beta\leq \max_{s\in [0,1]}S(\sigma(s))=S(\varphi)
\end{equation*}
for all $\varphi\in {B}_{1}\setminus\left\{0\right\}$  such that $J^{a,b}(\varphi)=0$. This implies that $\beta\leq d^{a,b}$. 
\end{proof}

Now we introduce the sets $\mathcal{K}^{a,b,\pm}$ defined by
\begin{align*}
\mathcal{K}^{a,b,+}&=\bigl\{ \varphi\in {B}_{1}: S(\varphi)<\beta,\quad  J^{a,b}(\varphi)\geq 0\bigl\}, \\
\mathcal{K}^{a,b,-}&=\bigl\{ \varphi\in {B}_{1}: S(\varphi)<\beta,\quad  J^{a,b}(\varphi)<0\bigl\}.  
\end{align*}

\begin{lemma} \label{L3}
The sets $\mathcal{K}^{a,b,\pm}$ are independent of
$(a,b)$ satisfying \eqref{Cd1}.
\end{lemma}
\begin{proof}
  Suppose first that in addition to \eqref{Cd1}, we have $2a+b(d-n)>0$. \\
It is clear that $\mathcal{K}^{a,b,-}$ is  open in
${B}_{1}$. Now we prove that $\mathcal{K}^{a,b,+}$ is
open. First, notice that by Lemma~\ref{L1}, if
$S(\varphi)<\beta$ and $J^{a,b}(\varphi)= 0$ then
$\varphi=0$. Moreover, using the fact that  a neighborhood of $0$ is
contained in $\mathcal{K}^{a,b,+}$ by Lemma \ref{LP1}, this
implies that $\mathcal{K}^{a,b,+}$ is  open in ${B}_{1}$. On
the other hand, since $2a+b(d-n)>0$ (notice that this implies
that $\|\varphi^{a,b}_{\lambda}\|_{B_{1}}\rightarrow0$ as
$\lambda\rightarrow -\infty$), using the same argument developed in
the proof of \cite[Lemma 2.9]{IbraMasmNaka2011}  it is not difficult
to show that $\mathcal{K}^{a,b,+}$ is connected. Thus, since
$0\in \mathcal{K}^{a,b,+}$ and
$\mathcal{K}^{a,b,+}\cup \mathcal{K}^{a,b,-}$ is
independent of $(a,b)$ (see Lemma \ref{LL2} (ii)), we infer that
$\mathcal{K}^{a,b,+}=\mathcal{K}^{a^{\prime},b^{\prime},+}$
for $(a,b)\neq (a^{\prime},b^{\prime})$ such that $2a+b(d-n)>0$ and
$2a^{\prime}+b^{\prime}(d-n)>0$. In particular we have
$\mathcal{K}^{a,b,-}=\mathcal{K}^{a^{\prime},b^{\prime},-}$.
\smallbreak

Now  assume that $2a+b(d-n)=0$.  We choose a sequence $\left\{(a_{j},b_j)
\right\}^{\infty}_{j=1}$ such that $(a_{j},b_{j})$ satisfies
\eqref{Cd1},  converges to
$(a,b)$, and $2a_{j}+b_{j}(d-n)>0$ for all $j$. Then
$J^{a_{j},b_{j}}\rightarrow J^{a,b}$ and we have 
\begin{equation*}
\mathcal{K}^{a,b,\pm}\subset\bigcup_{j\geq 1}\mathcal{K}^{a_{j},b_{j},\pm}.
\end{equation*}
By using the fact that the right side is independent of the parameter,
so is the left, which finishes the proof. 
\end{proof}

The following remark will be used in the sequel.
\begin{remark}\label{Remarp}
If $\varphi\neq 0$ satisfies $P(\varphi)=0$, then $S(\varphi)\geq \beta$. Indeed, we put $\varphi^{r}(x):=r^{\frac{d-n}{2}}\varphi(y, r z)$ for $r>0$. Then
\begin{equation*}
I(\varphi^{r})=r^{2}\|\nabla_{z}\varphi\|^{2}_{L^{2}}-r^{\sigma(d-n)}\|\varphi\|^{2\sigma+2}_{2\sigma+2}+K_{\varphi},
\end{equation*}
where 
\[K_{\varphi}=\|\nabla_{y}\varphi\|^{2}_{L^{2}}+ \|\varphi\|^{2}_{L^{2}}+ \|y\varphi\|^{2}_{L^{2}}>0.\]
From $P(\varphi)=0$, we see that
\[ I(\varphi^{r})=\left(\frac{(d-n)\sigma}{2(\sigma+1)}r^{2}-r^{\sigma(d-n)} \right)\|\varphi\|^{2\sigma+2}_{2\sigma+2}+K_{\varphi}.\]
Since $\sigma(d-n)>2$, there exists $r_{0}\in (0, \infty)$ such that $I(\varphi^{r_{0}})=0$. This implies that $S(\varphi^{r_{0}})\geq \beta$. Moreover, since  $\sigma(d-n)>2$ and  $\left.\partial_{r}S(\varphi^{r})\right|_{r=1}=\left(\frac{d-n}{2}\right)P(\varphi)=0$, it is not difficult to show that the function $r\mapsto S(\varphi^{r})$,  $r\in (0,\infty)$, attains its maximum at $r=1$. Therefore,
\[ S(\varphi) \geq S(\varphi^{r_{0}})\geq\beta.\]
\end{remark}

The next two lemmas will play an important role to get blow-up and global existence results.
\begin{lemma} \label{L4}
Let $\varphi\in \mathcal{K}^{+}$, then
\begin{equation*}
\frac{\sigma}{\sigma+1}\|\varphi\|^{2}_{B_{1}}\leq S(\varphi)\leq \frac{1}{2}\|\varphi\|^{2}_{B_{1}}
\end{equation*}
\end{lemma}
\begin{proof}
 From Lemma \ref{L3} we see that $I(\varphi)$ and $P(\varphi)$ have the same sign under the condition $S(\varphi)<\beta$. Since $\varphi\in \mathcal{K}^{+}$, we obtain $I(\varphi)\geq 0$, which implies that 
\begin{equation*}
\|\varphi\|^{2\sigma+2}_{L^{2\sigma+2}}\leq \|\varphi\|^{2}_{B_{1}}.
\end{equation*}
Therefore, 
\begin{equation*}
\frac{1}{2}\|\varphi\|^{2}_{B_{1}}\geq S(\varphi)=\frac{1}{2}\|\varphi\|^{2}_{B_{1}}-\frac{1}{2\sigma+2}\|\varphi\|^{2\sigma+2}_{L^{2\sigma+2}}\geq \frac{\sigma}{\sigma+1}\|\varphi\|^{2}_{B_{1}},
\end{equation*}
 and the proof is complete.
\end{proof}

\begin{lemma} \label{L5}
If $\varphi\in \mathcal{K}^{-}$, then 
\begin{equation*}
P(\varphi)\leq -\frac{4}{d-n}(\beta-S(\varphi))
\end{equation*}
\end{lemma}
\begin{proof}
We consider $\varphi\in \mathcal{K}^{-}$. We put $s(\lambda):=S(\varphi^{1,-2/(d-n)}_{\lambda})$ (see \eqref{Fm}). Then
\begin{align}\nonumber
s(\lambda)&=\frac{1}{2}\|\nabla _{y} \varphi\|^{2}_{L^{2}}+\frac{e^{4\lambda/(d-n)} }{2}\|\nabla _{z} \varphi\|^{2}_{L^{2}}+\frac{1}{2}\|\varphi\|^{2}_{L^{2}}+\frac{1}{2}\|y\varphi\|^{2}_{L^{2}}-\frac{e^{2\sigma\lambda} }{2\sigma+2}\| \varphi\|^{2\sigma+2}_{L^{2\sigma+2}},\\ \label{S11}
s^{\prime}(\lambda)&=\frac{2}{d-n}e^{4\lambda/(d-n)} \|\nabla _{z} \varphi\|^{2}_{L^{2}}-\frac{\sigma}{\sigma+1}e^{2\sigma\lambda} \| \varphi\|^{2\sigma+2}_{L^{2\sigma+2}},\\ \label{S22}
s^{\prime\prime}(\lambda)&=\frac{8}{(d-n)^{2}}e^{4\lambda/(d-n)}  \|\nabla _{z} \varphi\|^{2}_{L^{2}}-\frac{2\sigma^{2}}{\sigma+1}e^{2\sigma\lambda} \| \varphi\|^{2\sigma+2}_{L^{2\sigma+2}}.
\end{align}
Thus, we infer
\begin{equation}\label{DI}
s^{\prime\prime}(\lambda)=\frac{2\sigma}{\sigma+1}\left(\frac{2}{d-n}-\sigma\right)e^{2\sigma\lambda}\| \varphi\|^{2\sigma+2}_{L^{2\sigma+2}}+\frac{4}{d-n}s^{\prime}(\lambda)\leq\frac{4}{d-n}s^{\prime}(\lambda),
\end{equation}
where we have used that $\sigma>2/(d-n)$. Since $P(\varphi)<0$ and
$s^{\prime}(\lambda)>0$ for small $\lambda<0$, then by  continuity,
there exists $\lambda_{0}<0$ such that $s^{\prime}(\lambda)<0$ for any
$\lambda\in (\lambda_{0},0]$ and $s^{\prime}(\lambda_{0})=0$. Since
$s(\lambda_{0})\geq \beta$ (see Remark~\ref{Remarp}),  integrating \eqref{DI}  over
$(\lambda_{0},0]$, we obtain 
\begin{align*}
P(\varphi)&=s^{\prime}(0)=s^{\prime}(0)-s^{\prime}(\lambda_{0})\leq \frac{4}{d-n}(s(0)-s(\lambda_{0}))\leq \frac{4}{d-n}(S(\varphi)-\beta),
\end{align*}
hence the result. 
\end{proof}

\section{Criteria for Global well-posedness and blow-up}\label{S:2} 

In this section we prove our global well-posedness and blow-up result,
that is, Theorem~\ref{Th1} up to the scattering part.

\begin{proof}[Proof of Theorem~\ref{Th1}] 
(i) Let $u_{0}\in \mathcal{K}^{+}$. Since the energy and the mass are conserved, we have 
\begin{equation}\label{cf}
u(t)\in \mathcal{K}^{+}\cup \mathcal{K}^{-}, \quad
\text{for every $t$ in the existence interval.} 
\end{equation}
Here $u(t)$ is the corresponding solution of \eqref{GP} with
$u(0)=u_{0}$. Assume that there  exists $t_{0}>0$ such that
$u(t_{0})\in \mathcal{K}^{-}$. Since the map
$t\mapsto P(u(t))$ is continuous, there exists $t_{1}\in (0,t_{0})$
such that $P(u(t))<0$ for all $t\in (t_{1}, t_{0})$ and
$P(u(t_{1}))=0$. Thus, by Remark~\ref{Remarp} we see that if
$u(t_{{1}})\neq 0$, then  $S(u(t_{{1}}))\geq \beta$. However,
by \eqref{cf} we have $S(u(t_{{1}}))< \beta$, which is a
absurd. Therefore, $u(t)\in \mathcal{K}^{+}$ for every $t$ in
the existence interval. Now, by Lemma \ref{L4} we obtain that
$\|u(t)\|_{B_{1}}\sim S(u(t))<\beta$ for every $t$. By the
local theory (Lemma~\ref{lem:LWP}),  this implies that $u$ is global and $u(t)\in
\mathcal{K}^{+}$ for every $t\in \mathbb{R}$. The scattering
result will be shown in Section~\ref{SC1}. \\

(ii)  Similarly as above, we can show that  if $u_{0}\in
\mathcal{K}^{-}$, then $u(t)\in \mathcal{K}^{-}$ for
every $t$ in the interval $[0, T_{+})$. If $T_{+}<+\infty$, by
the local theory (Lemma~\ref{lem:LWP}), we have $\lim_{t\rightarrow T_{+}}\|\nabla
_{x}u(t)\|^{2}_{L^{2}}=+\infty$. On the other hand, if
$T_{+}=+\infty$ we prove that there exists $t_k\to \infty$ such
that $\lim_{t_{k}\to \infty}\|\nabla
_{x}u(t_k)\|^{2}_{L^{2}}=+\infty$ by contradiction: suppose
\begin{equation*}
k_{0}:=\sup_{t\ge 0}\|\nabla _{x}u(t)\|_{L^{2}}<+\infty.
\end{equation*}
Now we consider the localized virial identity and define
\begin{equation}\label{virial1}
V(t):=\int_{\mathbb{R}^{d}}\phi(z)|u(t,x)|^{2}dx, \quad x=(y,z)\in \mathbb{R}^{n}\times \mathbb{R}^{d-n}.
\end{equation}
Let $\phi\in C^{4}(\mathbb{R}^{d-n})$. If $\phi$ is a radial function (that is, $\phi(z)=\phi(|z|)$), by direct computations we have
\begin{align}\label{V1}
V^{\prime}(t)&=2\text{Im}\,\int_{\mathbb{R}^{d}}\nabla_{z} \phi\cdot\nabla_{z} u \overline{u},\\\label{V2}
V^{\prime\prime}(t)&=4\int_{\mathbb{R}^{d}}\text{Re}\,\<\nabla_{z}
                     \overline{u} , \nabla_{z}^{2}\phi \nabla_{z} u \>
-\frac{2\sigma}{\sigma+1}\int_{\mathbb{R}^{d}}\Delta_{z} \phi |u|^{2\sigma+2}-\int_{\mathbb{R}^{d}}\Delta_{z}^{2}\phi|u|^{2}.
\end{align}
Before continuing the proof of Theorem \ref{Th1} we first state the following result:
\begin{lemma} \label{L8}
Let $\eta>0$. Then for all $t\leq \eta R/(4k_{0}\|u_{0}\|_{L^{2}})$ we have
\begin{equation}\label{CII}
\int_{|z|\geq R}|u(t,x)|^{2}dx\leq \eta+o_{R}(1).
\end{equation}
\end{lemma}
\begin{proof}
Fix $R>0$,  and take $\phi$ in \eqref{virial1} such that
\begin{equation*}
\phi(r)=
\begin{cases} 
0,\quad 0\leq |z|\leq \frac{R}{2};\\
1,\quad |z|\geq {R},
\end{cases} 
\end{equation*}
where $r=|z|$ and
\begin{equation*}
0\leq \phi\leq 1, \quad  0\leq \phi^{\prime}\leq \frac{4}{R}.
\end{equation*}
From \eqref{V1} we infer that
\begin{align*}
 V(t)&=V(0)+\int^{t}_{0}V^{\prime}(s)ds\leq V(0)+t\|\phi^{\prime}\|_{L^{\infty}}\|u_{0}\|_{L^{2}}k_{0}\\
&\leq \int_{|z|\geq R/2}|u_{0}(x)|^{2}dx+\frac{4\|u_{0}\|_{L^{2}}k_{0}}{R}t.
\end{align*}
Moreover, Lebesgue's dominated convergence theorem yields
\begin{equation*}
\int_{|z|\geq R/2}|u_{0}(x)|^{2}dx=o_{R}(1),
\end{equation*}
and
\begin{equation*}
\int_{|z|\geq R}|u(t,x)|^{2}dx\leq V(t).
\end{equation*}
Therefore for given $\eta>0$, if
\begin{equation*}
t\leq \frac{\eta R}{4k_{0}\|u_{0}\|_{L^{2}}},
\end{equation*}
then we see that
\begin{equation*}
\int_{|z|\geq R}|u(t,x)|^{2}dx\leq \eta+o_{R}(1).
\end{equation*}
This concludes the proof of the lemma.
\end{proof}

Next we choose another function $\phi$ in \eqref{virial1} such that
\begin{equation*}
\phi(r)=
\begin{cases} 
r^{2},\quad 0\leq r\leq {R};\\
0,\quad r\geq 2R,
\end{cases} 
\end{equation*}
with
\begin{equation*}
0\leq \phi\leq r^{2}, \quad   \phi^{\prime\prime}\leq2, \quad
\phi^{(4)}\leq \frac{4}{R^{2}}. 
\end{equation*}
By \eqref{V1}, $V^{\prime}(t)$ and $V^{\prime\prime}(t)$ can be rewritten as
\begin{align}\label{LV1}
V^{\prime}(t)&=2\text{Im}\,\int_{\mathbb{R}^{d}}\frac{\phi^{\prime}(r)}{r}z\cdot\nabla_{z} u \overline{u},\\\label{LV2}
V^{\prime\prime}(t)&=4\int_{\mathbb{R}^{d}}\frac{\phi^{\prime}}{r}|\nabla_{z} u|^{2} +4\int_{\mathbb{R}^{d}}\left(\frac{\phi^{\prime\prime}}{r^{2}}-\frac{\phi^{\prime}}{r^{3}}\right)|z\cdot \nabla_{z}u|^{2}\\\nonumber
&-\frac{2\sigma}{\sigma+1}\int_{\mathbb{R}^{d}}\(\phi^{\prime\prime}+(d-n-1)\frac{\phi^{\prime}}{r}\) |u|^{2\sigma+2}-\int_{\mathbb{R}^{d}}\Delta_{z}^{2}\phi|u|^{2}\\\label{LV3}
&=4(d-n)P(u)+R_{1}+R_{2}+R_{3},
\end{align}
where 
\begin{equation}\label{R11}
\begin{aligned}
 &R_{1}=4\int_{\mathbb{R}^{d}}\left(\frac{\phi^{\prime\prime}}{r}-2\right)|\nabla_{z} u|^{2} +4\int_{\mathbb{R}^{d}}\left(\frac{\phi^{\prime\prime}}{r^{2}}-\frac{\phi^{\prime}}{r^{3}}\right)|z\cdot \nabla_{z}u|^{2}\\
&R_{2}=-\frac{2\sigma}{\sigma+1}\int_{\mathbb{R}^{d}}\(\phi^{\prime\prime}+(d-n-1)\frac{\phi^{\prime}}{r}-2(d-n)\) |u|^{2\sigma+2},\\
&R_{3}=-\int_{\mathbb{R}^{d}}\Delta_{z}^{2}\phi|u|^{2}.
\end{aligned}
\end{equation}
First we show that $R_{1}\leq 0$. Indeed, we can decompose
$\mathbb{R}^{d}$ into
\begin{equation*}
  \R^d = \underbrace{\left\{\phi^{\prime\prime}/r^{2}-\phi^{\prime}/r^{3}\leq
    0\right\}}_{=: \Omega_1}\cup \underbrace{\left\{\phi^{\prime\prime}/r^{2}-\phi^{\prime}/r^{3}>  0\right\}}_{=: \Omega_2}.
\end{equation*}
On $\Omega_1$, since $\phi^{\prime}\leq 2r$,
\begin{equation*}
  4\int_{\Omega_1}\left(\frac{\phi^{\prime\prime}}{r}-2\right)|\nabla_{z}
  u|^{2}
  +4\int_{\Omega_1}\left(\frac{\phi^{\prime\prime}}{r^{2}}-\frac{\phi^{\prime}}{r^{3}}\right)|z\cdot
  \nabla_{z}u|^{2} \le 0.
\end{equation*}
On $\Omega_2$, 
\begin{equation*}
 \int_{\Omega_2}\left(\frac{\phi^{\prime\prime}}{r}-2\right)|\nabla_{z}
  u|^{2}
  +\int_{\Omega_2}\left(\frac{\phi^{\prime\prime}}{r^{2}}-\frac{\phi^{\prime}}{r^{3}}\right)|z\cdot
  \nabla_{z}u|^{2} \leq \int_{\mathbb{R}^{d}}(\phi^{\prime\prime}-2)|\nabla_{z}u|^{2}dx\leq 0.
\end{equation*}
Secondly, notice that $\text{supp}\chi\subset [R,\infty)$, where
\begin{equation*}
\chi(r)=\left|\phi^{\prime\prime}(r)+(d-n-1)\frac{\phi^{\prime}(r)}{r}-2(d-n)\right|.
\end{equation*}
For $2\si+2<q<\tfrac{2d}{d-2}$, there exists
$0<\theta< 1$ such that
$\frac{1}{2\sigma+2}=\frac{1-\theta}{q}+\frac{\theta}{2}$, and 
\begin{equation}\label{A2}
R_{2}^{\frac{1}{2\si+2}}\lesssim \|
u\|_{L^{2\sigma+2}(|z|>R)}
\le\|u\|^{1-\theta}_{L^{q}(|z|>R)}\|u\|^{\theta}_{L^{2}(|z|>R)}\lesssim
k^{1-\theta}_{0}\| u\|^{\theta}_{L^{2}(|z|> R)}.
\end{equation}
Finally, 
\begin{equation}\label{A3}
R_{3}\leq CR^{-2}\| u\|^{2}_{L^{2}(|z|> R)}.
\end{equation}
Combining \eqref{LV3}, \eqref{A2} and \eqref{A3} we obtain
\begin{equation}\label{A4}
V^{\prime\prime}(t)\leq
4(d-n)P(u(t))+C\|u\|^{(2\si+2)\theta}_{L^{2}(|z|> R)}+ CR^{-2}\| u\|^{2}_{L^{2}(|z|> R)} ,
\end{equation}
where $C>0$  depends only on $\|u_{0}\|_{L^{2}}$, $k_{0}$ and $\sigma$.  By Lemma \ref{L8} we obtain that for all $t\leq T:= \eta R/(4k_{0}\|u_{0}\|^{2}_{L^{2}})$,
\begin{equation*}
V^{\prime\prime}(t)\leq 4(d-n)P(u(t))+C\(\eta^{(2\si+2)\theta}+\eta^2+o_{R}(1)\),
\end{equation*}
and since $u(t)\in \mathcal{K}^{-}$, Lemma~\ref{L5} yields
$P(u(t))\leq-\tfrac{4}{d-n}(\beta-S(u_{0}))<0$. Thus, 
\begin{equation}\label{DDa}
V^{\prime\prime}(t)\leq -16(\beta-S(u_{0})) +C\(\eta^{(2\si+2)\theta}+\eta^2+o_{R}(1)\).
\end{equation}
Integrating \eqref{DDa} from $0$ to $T$ we infer
\begin{align*}
V(T)\leq V(0)+V^{\prime}(0)T +\(-16(\beta-S(u_{0}))
  +C\(\eta^{(2\si+2)\theta}+\eta^2+o_{R}(1)\)\)T^2 .
\end{align*}
Choosing $\eta$ sufficiently small and taking $R$ large enough, it follows that for $T= \eta R/(4k_{0}\|u_{0}\|_{L^{2}})$ we have
\begin{equation*}
-16(\beta-S(u_{0}))
  +C\(\eta^{(2\si+2)\theta}+\eta^2+o_{R}(1)\) < -8 (\beta-S(u_{0})),
\end{equation*}
and 
\begin{equation*}
V(T)\leq V(0)+V^{\prime}(0)\frac{\eta R}{4k_{0}\|u_{0}\|_{L^{2}}}+\mu_{0}R^{2},
\end{equation*}
where
\begin{equation*}
\mu_{0}=-\frac{(\beta-S(u_{0}))\eta^2}{2 k_0^2\|u_0\|_{L^2}^2}<0.
\end{equation*}
Next notice that we have $V(0)\leq o_{R}(1)R^{2}$ and $V^{\prime}(0)\leq o_{R}(1)R$. Indeed,
\begin{align*}
V(0)&\leq \int_{|z|< \sqrt{R}}|z|^{2}|u_{0}(x)|^{2}dx+\int_{\sqrt{R}<|z|< 2R}|z|^{2}|u_{0}(x)|^{2}dx\\
&\leq R\|u_{0}\|^{2}_{L^{2}}+4R^{2}\int_{|z|>\sqrt{R}}|u_{0}(x)|^{2}dx\\
&= o_{R}(1)R^{2}.
\end{align*}
Moreover,
\begin{align*}
V^{\prime}(0)&\leq \int_{|z|<
               \sqrt{R}}|z||u_{0}|\lvert\nabla_{z}u_{0}\rvert dx+\int_{\sqrt{R}<|z|< 2R}|z||u_{0}||\nabla_{z}u_{0}| dx\\
&\leq \sqrt{R}\|u_{0}\|^{2}_{H^{1}}+2R\int_{|z|>\sqrt{R}}|u_{0}||\nabla_{z}u_{0}| dx\\
&= o_{R}(1)R.
\end{align*}
Thus we get
\begin{equation*}
V(T)\leq (o_{R}(1)+\mu_{0})R^{2},
\end{equation*}
and for $R$ sufficiently enough, $o_{R}(1)+\mu_{0}<0$, which is a contradiction since $V(T)>0$. The proof of Theorem \ref{Th1} is now
complete.
\end{proof}

\section{Proof of the scattering result}\label{SC1}
In Section \ref{S:2} we showed that if $u_{0}\in \mathcal{K}^{+}$, then the solution is global and belongs  to $\mathcal{K}^{+}$ for all $t\in \mathbb{R}$. In this section we show that under this condition, the solution scatters in $B_{1}$.

\subsection{Small data scattering}\label{sds1}

We begin with some lemmas complementing the results of
Section~\ref{sec:scattering1}. Recall that the indices considered here
were introduced in Section~\ref{sec:indices}. The first lemma covers
both the Cauchy problem ($t_0\in \R$) and the existence of wave
operators ($|t_0|=\infty$).

\begin{lemma}[Small data scattering]\label{ssp} Suppose
  $\tfrac{2}{d-n}\le \si<\tfrac{2}{d-2}$, $\lambda\in \{-1,1\}$. Let $\varphi\in
  B_{1}$. There exists $\delta>0$ 
  such that  if
  $\|e^{-itH}\varphi\|_{\ell_{\gamma}^{{p}}L^{q}L^{r}}\leq
  \delta$,  then for all $t_0\in [-\infty,\infty]$, the solution $u$ to
  \begin{equation}\label{eq:Duhamel-gen}
    u(t) = e^{-itH} \varphi -i\lambda \int_{t_0}^t e^{-i(t-s)H}\(|u|^{2\si}u\)(s)ds
  \end{equation}
  is global for both positive and negative times, and
  satisfies 
\[  \|u\|_{\ell_{\gamma}^{{p}}L^{q}L^{r}}
  \leq
  2 \|e^{-itH}\varphi\|_{\ell_{\gamma}^{{p}}L^{q}L^{r}}.  \]
There exists $\nu>0$ such that if
  $\|\varphi\|_{B_1}\leq
  \nu$,  then $\|e^{-itH}\varphi\|_{\ell_{\gamma}^{{p}}L^{q}L^{r}}\leq
  \delta$, and 
  for all $t_0\in [-\infty,\infty]$, the solution $u$ to
  \eqref{eq:Duhamel-gen}
  is global for both positive and negative times, and
  satisfies 
\[  \|u\|_{B_1}
  \leq
  2 \|\varphi\|_{B_1}.  \]
\end{lemma}
\begin{proof}
  Denote by
  \begin{equation*}
    \Phi(u)(t) :=  e^{-itH} \varphi -i\lambda \int_{t_0}^t
    e^{-i(t-s)H}\(|u|^{2\si}u\)(s)ds. 
  \end{equation*}
  First, consider
  \begin{align*}
    X=\Big\{& u \in C(\R;B_1);\ \|u\|_{\ell_{\gamma}^{{p}}L^{q}L^{r}}
  \leq
    2 \|e^{-itH}\varphi\|_{\ell_{\gamma}^{{p}}L^{q}L^{r}},\\
   &\forall A\in \{{\rm
    Id},A_1,A_2,\nabla_z\},\ \|Au\|_{\ell^{p_0}_{\gamma}L^{q_0}L^{p_0}}\le
     2 C_0\|A \varphi\|_{L^2} 
\Big\},
\end{align*}
 where $C_0$ is the constant associated to the homogeneous Strichartz estimate
 \eqref{Sz}  ($F=0$) in the case $(p_1,q_1,r_1)=(p_0,q_0,r)$.
Let $u\in X$. In view of the
inhomogeneous Strichartz estimates (Lemmas~\ref{Ls22} and
\ref{lem:indices}), and since
\begin{equation*}
p=(2\sigma+1)\tilde{p}',\quad   q=(2\sigma+1)\tilde{{q}}',\quad 
 r=(2\sigma+1)r',
\end{equation*}
we have
\begin{align*}
  \|\Phi(u)\|_{\ell^p_{\gamma}L^qL^r}& \le
  \|e^{-itH}\varphi\|_{\ell_{\gamma}^{p}L^{q}L^{r}} +
 C \left\| |u|^{2\si}u\right\|_{\ell^{\tilde p'}_{\gamma}L^{\tilde q'}L^{r'}}\\
  &\le
      \|e^{-itH}\varphi\|_{\ell_{\gamma}^{p}L^{q}L^{r}} + C\|
        u\|_{\ell^p_\gamma L^qL^r}^{2\si+1}.
\end{align*}
For $\delta>0$ sufficiently small, the right hand side
does not exceed $2\delta$.
\smallbreak

Reproducing the estimates of the
proof of Lemma~\ref{Ls2}, for  $A\in \{{\rm Id},A_1,A_2,\nabla_z\}$,
\begin{equation*}
  \|A\Phi(u)\|_{\ell^{p_0}_\gamma
  L^{q_0}L^{p_0}}\le C_0 \|A\varphi\|_{L^2} + C_1
  \|u\|_{\ell^p_\gamma L^qL^r}^{2\si}\|Au\|_{\ell^{p_0}_\gamma
L^{q_0}L^{p_0}}. 
\end{equation*}
Up to choosing $\delta>0$ smaller, we infer
\begin{equation*}
   \|A\Phi(u)\|_{\ell^{p_0}_\gamma
  L^{q_0}L^{p_0}}\le 2 C_0 \|A\varphi\|_{L^2} ,
\end{equation*}
and so $\Phi$ maps $X$ to itself. We equip $X$ with the metric
\begin{equation*}
  d(u,v) = \|u-v\|_{\ell^p_\gamma L^qL^r}
\end{equation*}
which makes it a complete space (see e.g. \cite{CB}). We then have
\begin{align*}
  \left\| \Phi(u)-\Phi(v)\right\|_{\ell^p_\gamma L^qL^r}& \lesssim
  \left\| |u|^{2\si}u-|v|^{2\si}v\right\|_{\ell^{\tilde p'}_\gamma
                                                          L^{\tilde q'}L^{r'}} \\
  &\lesssim  \left\|
  \(|u|^{2\si}+|v|^{2\si}\)(u-v)\right\|_{\ell^{\tilde p'}_\gamma 
   L^{\tilde q'}L^{r'}} \\
  &\lesssim \( \|u\|_{\ell^p_\gamma L^qL^r}^{2\si} +
  \|v\|_{\ell^p_\gamma L^qL^r}^{2\si}         \)\|u-v\|_{\ell^p_\gamma
    L^qL^r},                    \end{align*}
so contraction follows, up to choosing $\delta>0$ smaller, hence the
first part of the lemma.

For the second part, note that in view of Lemma~\ref{axc}, for $\nu>0$
sufficiently small, $\|e^{-itH}\varphi\|_{\ell_{\gamma}^{{p}}L^{q}L^{r}}\leq
  \delta$, and we may use the first part of the lemma. Strichartz
  estimates also yield, for $A\in \{{\rm Id},A_1,A_2,\nabla_z\}$,
\begin{equation*}
  \|Au\|_{L^\infty_tL^2}\le  \|A\varphi\|_{L^2} + C_2
  \|u\|_{\ell^p_\gamma L^qL^r}^{2\si}\|Au\|_{\ell^{p_0}_\gamma
L^{q_0}L^{p_0}}. 
\end{equation*}
Up to choosing $\delta>0$ smaller, we infer
\begin{equation*}
  \|Au\|_{L^\infty_tL^2}\le  2\|A\varphi\|_{L^2} ,
\end{equation*}
hence the second part of the lemma, from \eqref{eq:vector-norm}.
\end{proof}

We now go back to the focusing case, $\lambda=-1$.

\begin{lemma}[Wave operators for not so small data]\label{Wape} Suppose
  $\tfrac{2}{d-n}\le \si<\tfrac{2}{d-2}$. Let
  $\psi\in B_{1}$ such that
\begin{equation}\label{Imi111}
\frac{1}{2}\|\nabla_{x} \psi\|^{2}_{L^{2}}+\frac{1}{2}\|y \psi\|^{2}_{L^{2}}+\frac{1}{2}\|\psi\|^{2}_{L^{2}}<\beta,
\end{equation}
where $\beta$ is given by \eqref{mpp}.  Then there exists $u_{0}\in
\mathcal K^{+}$ such that the corresponding solution $u(t)$ of
\eqref{GP} with $u(0)=u_{0}$ satisfies
\begin{equation*}
\|e^{itH}u(t)-\psi\|_{B_{1}}\Tend t\infty 0.
\end{equation*}
\end{lemma}
\begin{proof}
Consider the integral equation
\begin{equation}\label{Iee}
u(t)=e^{-itH}\psi-i\int^{+\infty}_{t}e^{-i(t-s)H}(|u|^{2\sigma}u)(s)ds=:\Phi(u)(t).
\end{equation}
We first construct a
solution defined on $[T,\infty)$ for $T\gg 1$ by a fixed point
argument similar to the one employed in the proof of Lemma~\ref{ssp}. Introduce
\begin{align*}
  X_T=\Big\{ &u\in C([\pi (T-1),\infty);B_1);\ 
  \|u\|_{\ell^p_{\gamma\ge
      T}L^qL^r}\le 2\|e^{-itH}\psi\|_{\ell_{\gamma\ge
               T}^{{p}}L^{q}L^{r}} ,\\
  &\sum_{A\in \{{\rm
    Id},A_1,A_2,\nabla_z\}}\|Au\|_{\ell^{p_0}_{\gamma\ge
    T}L^{q_0}L^{p_0}}\le 2 C_0\|\psi\|_{B_1}
\Big\},
\end{align*}
where $C_0$ is the constant associated to the Strichartz estimate
\eqref{Sz} in the case $(p_1,q_1,r_1)=(p_0,q_0,r)$. 
By Lemma~\ref{axc},
$\|e^{-itH}\psi\|_{\ell_{\gamma\ge T}^{{p}}L^{q}L^{r}}\to 0$ as $T\to
\infty$. Therefore, choosing $T$ sufficiently large is equivalent to
requiring $\delta$ sufficiently small in the proof of
Lemma~\ref{ssp}. The proof is then the same, and we omit it.
We must now prove that the solution $u$ is
defined for all time.
\smallbreak

Since $e^{-itH}$ conserves the linear energy and $\|
e^{-itH}\psi\|_{L^{2\sigma+2}}^{2\sigma+2}\to 0$ as $
t\rightarrow\infty$ (see Lemma~\ref{lem:dispLp}), we have
\begin{gather*}
S(u(t))=\lim_{t\rightarrow\infty}S(e^{-itH}\psi)=\frac{1}{2}\|\nabla\psi\|_{L^{2}}^{2}+\frac{1}{2}\| y\psi\|_{L^{2}}^{2}+\frac{1}{2}\|\psi\|_{L^{2}}^{2}<\beta,\\
\lim_{t\rightarrow\infty}I(u(t))=\lim_{t\rightarrow\infty}\(\|
e^{-itH}\psi\|_{B_{1}}^2
-\|e^{-itH}\psi\|_{L^{2_{\sigma}+2}}^{2\sigma+2}\)=\|\psi\|_{B_{1}}^2 >0.
\end{gather*}
Thus, there exists $t^{*}$ sufficiently large such that $u(t^{*})\in
\mathcal{K}^{+}$. By using the fact that
$\mathcal{K}^{+}$ is invariant by the flow of \eqref{GP} we
obtain that $u(0)=u_{0}\in \mathcal{K}^{+}$.

By Strichartz estimates, like in the proof of Lemma~\ref{Ls2},
\begin{align*}
\|e^{itH}u(t)-\psi\|_{B_1}&\sim \sum_{A\in \{{\rm
                            Id},A_1,A_2,\nabla_z\}}\|A(t)u(t) -A(0)\psi\|_{L^2}\\
  &\lesssim \sum_{A\in \{{\rm
                            Id},A_1,A_2,\nabla_z\}}\left\|
    A(|u|^{2\si}u)\right\|_{\ell^{p_0'}_{\gamma\gtrsim t}L^{q_0'}L^{r'}
    }\\
  &\lesssim \sum_{A\in \{{\rm   Id},A_1,A_2,\nabla_z\}}\|u
    \|^{2\si}_{\ell^{p}_{\gamma\gtrsim 
    t}L^{q}L^{r}}\|Au\|_{\ell^{p_0}_{\gamma\gtrsim
    t}L^{q_0}L^{r}}\Tend {t}\infty 0,
\end{align*}
hence the lemma.
\end{proof}

\subsection{Perturbation lemma and linear profile decomposition}

We begin with the following result
\begin{lemma}[Perturbation lemma]\label{lpt}
  Suppose
  $\tfrac{2}{d-n}\le \si<\tfrac{2}{d-2}$. 
Let $\tilde{u}\in C([0, \infty); B_{1})$ be the solution of 
\begin{equation}\label{GP12}
i\d_t \tilde{u}-H\tilde{u}+|\tilde{u}|^{2\sigma}\tilde{u}=e,
\end{equation}
where $e\in L^{1}_{\rm loc}([0,\infty); B_{-1})$.  Given $A>0$, there exist $C(A)>0$ and $\eps(A)>0$ such that if ${u}\in C([0, \infty); B_{1})$ is a solution of \eqref{GP}, and if
\begin{equation}\label{Cia}
\begin{split}
&\|\tilde{u}\|_{\ell_{\gamma}^{{p}}L^{q}L^{r}}\leq A, \quad \|e\|_{\ell_{\gamma}^{{\tilde{p}'}}L^{\tilde{q}'}L^{r'}}\leq \eps\leq\eps(A),\\
&\|e^{-itH}({u}(0)-\tilde{u}(0))\|_{\ell_{\gamma}^{{p}}L^{q}L^{r}}\leq \eps\leq\eps(A),
\end{split}
\end{equation}
then $\|{u}\|_{\ell_{\gamma}^{{p}}L^{q}L^{r}}\leq C(A)<\infty$.
\end{lemma}
\begin{proof}
We omit the proof, which can be obtained by suitably adapting the
argument of \cite[Proposition 4.7]{FangXieCaze2011}, thanks to the
same Strichartz estimates as in the proof of Lemma~\ref{ssp}.
\end{proof}

We need the following linear profile decomposition, which is crucial
in the construction of a minimal blow-up solution. This is where the
assumption $\si\ge \tfrac{2}{d-n}$ becomes $\si> \tfrac{2}{d-n}$, in
order to prove \eqref{Eh} below. 
\begin{proposition}[Linear profile decomposition]\label{Ls3}
Suppose
  $\tfrac{2}{d-n}< \si<\tfrac{2}{d-2}$.  Let $\left\{\phi_{k}\right\}^{\infty}_{k=1}$ be a uniformly bounded
sequence in ${B}_{1}$. Then, up to subsequence, the following
decomposition holds.
\begin{equation*}
\phi_{k}(x)=\sum^{M}_{j=1}e^{it^{j}_{k}H}\psi^{j}\(y, z-z^j_{k}\)+W^{M}_{k}(x)\quad \text{for all $M\geq1$,}
\end{equation*}
where $t^{j}_{k}\in \mathbb{R}$, $z^{j}_{k}\in \mathbb{R}^{d-n}$,
$\psi^{j}\in {B}_{1}$ are such that: 
\begin{itemize}
  \item Orthogonality of the parameters
	\begin{equation}\label{El1}
|t^{j}_{k}-t^{\ell}_{k}|+|z^{j}_{k}-z^{\ell}_{k}|\Tend k \infty \infty, \quad \text{for  $j\neq \ell$,}
\end{equation}
	\item Asymptotic smallness property:
\begin{equation}\label{El2}
  \lim_{M\rightarrow\infty}\(\lim_{k\rightarrow\infty}\|e^{-itH}
    W^{M}_{k}\|_{\ell_{\gamma}^{{p}}L^{q}L^{r}}\)=0.  
	\end{equation}
\item Orthogonality in norms: for any fixed $M$ we have
\begin{align}\label{El23}
\|\phi_{k}\|^{2}_{L^{2}}&=\sum^{M}_{j=1}\|\psi^{j}\|^{2}_{L^{2}}+\|W^{M}_{k}\|^{2}_{L^{2}}+o_{k}(1),\\\label{El4}
\|\phi_{k}\|^{2}_{{\dot{B}_{1}}}&=\sum^{M}_{j=1}\|\psi^{j}\|^{2}_{{\dot{B}_{1}}}+\|W^{M}_{k}\|^{2}_{{\dot{B}_{1}}}+o_{k}(1).
\end{align}
\end{itemize}
Furthermore, we have
\begin{equation}\label{Lps}
\|\phi_{k}\|^{2\sigma+2}_{L^{2\sigma+2}}=\sum^{M}_{j=1}\|e^{it^{j}_{k}H}\psi^{j}\|^{2\sigma+2}_{L^{2\sigma+2}}+\|W^{M}_{k}\|^{2\sigma+2}_{L^{2\sigma+2}}+o_{k}(1)\quad \text{for all $M\geq1$.}
 \end{equation}
In particular, for all  $M\geq1$
\begin{align}\label{Sg}
S(\phi_{k})&=\sum^{M}_{j=1}S\left( e^{it^{j}_{k}H}\psi^{j}\right)+S(W^{M}_{k})+o_{k}(1)\\    \label{Ig}
I(\phi_{k})&=\sum^{M}_{j=1}I\left( e^{it^{j}_{k}H}\psi^{j}\right)+I(W^{M}_{k})+o_{k}(1).
\end{align}
\end{proposition}
We note that cores are present only in the $z$-variable, not in the
$y$-variable. This is so because the partial harmonic potential has a
confining effect, hence in $y$, the situation is similar to the radial
setting (as in \cite{KenigMerle2006,HolmerRoudenko2008}). 
\begin{proof}
First, we show that there exist $\theta\in (0,1)$ such that 
\begin{equation} \label{Eh}
\|e^{-itH}f \|_{\ell_{\gamma}^{p}L^{q}L^{r}}
\lesssim
\|f\|_{B_{1}}^{1-\theta}\|e^{-itH}f\|^{\theta}_{L^{\infty}_{t}L^{r}_{x}},\quad
\forall f\in B_1.
\end{equation}
Indeed, from \eqref{auxn1} we have
\begin{equation*}
   \| e^{-itH}f\|_{\ell_{\gamma}^pL^q L^r}\lesssim \|\chi(\cdot
   -\gamma\pi) e^{-itH}f\|_{\ell^{p}_\gamma
     W^{s,q_0}(\mathbb{R};L^r(\mathbb{R}^d))}.
\end{equation*}
 Since $\sigma>\tfrac{2}{d-n}$, we have $p_{0}<p$ and thus there
exists $\alpha\in (0, 1)$ such that 
\begin{align}\nonumber
   \| e^{-itH}f\|_{\ell_{\gamma}^pL^q L^r}&\lesssim \|\chi(\cdot
   -\gamma\pi) e^{-itH}f\|_{\ell^{p_0}_\gamma
     W^{s,q_0}(\mathbb{R};L^r(\mathbb{R}^d))}^\alpha \times\\\label{De11}
		&\|\chi(\cdot
   -\gamma\pi) e^{-itH}f\|_{\ell^{\infty}_\gamma
     W^{s,q_0}(\mathbb{R};L^r(\mathbb{R}^d))}^{1-\alpha}.
\end{align}
By the  homogeneous Strichartz estimate we get, like in the proof of
Lemma~\ref{axc}, 
\begin{equation}\label{De112}
  \begin{aligned}
  \|\chi(\cdot
   -\gamma\pi) e^{-itH}f\|_{\ell^{p_0}_\gamma
     W^{s,q_0}(\mathbb{R};L^r(\mathbb{R}^d))}&\lesssim \|\chi(\cdot
   -\gamma\pi) e^{-itH}H^{s}f\|_{\ell^{p_0}_\gamma
     L^{q_0}L^r}\\
   &\lesssim \|f\|_{B_{2s}}\lesssim
   \|f\|_{B_1}. 
  \end{aligned}
  \end{equation}
 Next we interpolate between Sobolev spaces in time, there is $\eta\in
 (0,1)$ such that
 \begin{gather}\nonumber
   \|\chi(\cdot
   -\gamma\pi) e^{-itH}f\|_{ W^{s,q_0}(\mathbb{R};L^r(\mathbb{R}^d))}\leq\\\label{De113}
	 \|\chi(\cdot
   -\gamma\pi) e^{-itH}f\|_{ W^{1/2,q_0}(\mathbb{R};L^r(\mathbb{R}^d))}^{1-\eta}\|\chi(\cdot
   -\gamma\pi) e^{-itH}f\|_{ L^{q_0}(\mathbb{R};L^r(\mathbb{R}^d))}^{\eta}.
 \end{gather}
 Moreover, we have
 \begin{align}\nonumber
   \|\chi(\cdot
   -\gamma\pi) e^{-itH}f\|_{\ell_{\gamma}^\infty
   W^{1/2,q_0}(\mathbb{R};L^r(\mathbb{R}^d))}&
  \lesssim \|e^{-itH}
   H^{1/2}f\|_{\ell_{\gamma}^\infty L^{q_0}L^r}\\ \notag
   &\lesssim \|e^{-itH}
   H^{1/2}f\|_{\ell_{\gamma}^{p_0}L^{q_0}L^r}\\
	&\lesssim \|H^{1/2}f\|_{L^2}= \|f\|_{B_1},\label{De114}
 \end{align}
and
\begin{equation}\label{Esu}
\|\chi(\cdot
   -\gamma\pi) e^{-itH}f\|_{\ell_{\gamma}^{\infty} L^{q_0}(\mathbb{R};L^r(\mathbb{R}^d))}\lesssim \|e^{-itH}
  f\|_{L_{t}^\infty L_{x}^r}.
\end{equation}
Combining \eqref{De11}, \eqref{De112}, \eqref{De113}, \eqref{De114} and \eqref{Esu} we obtain \eqref{Eh}.  

Since we will know that $\|W^{M}_{k}\|_{B_{1}}$  is uniformly bounded, then
to prove \eqref{El2}, it will suffice  to show that
\begin{equation*}
\lim_{M\rightarrow\infty}\(\lim_{k\rightarrow\infty}\|e^{-itH}W^{M}_{k}\|_{L_{t}^{\infty}L_{x}^{r}}\)=0. 
	\end{equation*}
We can then essentially repeat the proof of
\cite[Theorem~5.1]{FangXieCaze2011}, which generalized
\cite[Lemma~2.1]{DuyHolmerRoude2008}. Note that in the confined
variable 
$y$, the situation is similar to the radial setting without potential
(see e.g. \cite[Lemma~5.2]{HolmerRoudenko2008}),
this is why no core in $y$ will appear, only cores in $z$ (denotes by
 $z_k^j$), due to the translation invariance in $z$. Another technical
 difference is that Sobolev spaces $H^s$ have to be replaced with the
 spaces $B_s$ defined in the introduction.  Unlike
in the case without 
potential, $e^{-itH}$ does not commute with the convolution with
Fourier multipliers, nor is unitary on $\dot H^s$, and this imposes some extra
modification in the analysis.  
\smallbreak

\noindent \textbf{Step 1.} First we construct $t^{1}_{k}$, $z^{1}_{k}$,
$\psi^{1}$ and $W^{1}_{k}$. This is done by adapting
\cite[Lemma~5.2]{FangXieCaze2011}. By assumption, there exists  a positive
constant $\Lambda>0$ such that $\|\phi_{k}\|_{B_{1}}\leq \Lambda$. We infer
$\|e^{-itH}\phi_{k}\|_{L^{\infty}_{t}L^{r}_{x}}\lesssim
\|e^{-itH}\phi_{k}\|_{L^{\infty}_{t}B_1}= \|\phi_{k}\|_{{B}^{1}}\leq \Lambda$. Passing to a subsequence, we define
\begin{equation}\label{P11}
A_{1}:=\lim_{k\rightarrow\infty}\|e^{-itH}\phi_{k}\|_{L^{\infty}_{t}L^{r}_{x}}.
\end{equation}
If $A_{1}=0$, we set $\psi^{j}=0$ and $W^{1}_{k}=\phi_{k}$ for all
$k\geq1$.
We now suppose that $A_{1}>0$. We introduce a real-valued, radially
symmetric function $\varphi\in C_0^\infty(\R^d)$ supported in
$\{\xi\in \R^d;\ |\xi|\le 2\}$, such that
$\varphi(\xi)=1$ for 
$ |\xi|\leq 1$. For $N>1$ (to be chosen
later), in the same fashion as in \cite{HaniThoman2016}, define the
operator
\begin{equation*}
  P_{\le N}=\varphi\( \frac{-\Delta_y +| y|^2}{N^2}\)
\varphi\(\frac{-\Delta_z}{N^2}\),
\end{equation*}
where the first operator is to be understood as a spectral cut-off,
since the harmonic oscillator possesses an eigenbasis consisting of
Hermite functions, and the second operator is a Fourier (in $z$)
cut-off. By considering this operator instead of a Fourier cut-off in
$x$ (presented as a convolution in
\cite{HolmerRoudenko2008,FangXieCaze2011}), we gain the commutation
property 
\begin{equation*}
  [e^{-itH},P_{\le N}]=0.
\end{equation*}
Also, since $-\Delta_y +|y|^2$ and $-\Delta_z$ commute and are
positive operators, 
we have for $s\in (0,1)$ and $f\in B_1$,
\begin{align*}
\|f-P_{\le N} f\|_{B_{s}}
&=\|(1-P_{\le N})H^{\frac{s-1}{2}}H^{\frac{1-s}{2}} f\|_{B_{s}}\le
 \frac{1}{N^{1-s}}\|f\|_{B_{1}}. 
\end{align*}
In view of the Sobolev embedding $\dot H^s(\R^d)\hookrightarrow
L^{2\si+2}(\R^d)$ 
with $s=\tfrac{d\si}{2\si+2}$, and of the fact that $e^{-itH}$ is
bounded on $B_s$,
\begin{equation}\label{Ni1}
  \begin{aligned}
    \|e^{-itH}\phi_{k}-e^{-itH}P_{\le N}
\phi_{k}\|_{L^{\infty}_{t}L^{r}_{x}}&\lesssim \|e^{-itH}\phi_{k}-e^{-itH}P_{\le N}
\phi_{k}\|_{L^{\infty}_{t}\dot H^s_{x}} \\
& \lesssim \|e^{-itH}\phi_{k}-e^{-itH}P_{\le N}
\phi_{k}\|_{L^{\infty}_{t}B_{s}}  \\
&\lesssim \|\phi_{k}-P_{\le N}
\phi_{k}\|_{L^{\infty}_{t}B_{s}}    \le C_0 \frac{\Lambda}{N^{1-s}} \leq \frac{A_{1}}{2} ,
 \end{aligned}
\end{equation}
with $N = \(\tfrac{2  C_0\Lambda}{A_1}\)^{1/(1-s)}+1$. 
 It follows by \eqref{Ni1} that for $k$ large,
\begin{equation}\label{Pf1}
\|P_{\le N}  e^{-itH}\phi_{k}\|_{L^{\infty}_{t}L^{r}_{x}}\geq \frac{1}{4}A_{1}.
\end{equation}
 Moreover,  by interpolation we have
\begin{align*}
\|P_{\le N} e^{-itH}\phi_{k}\|_{L^{\infty}_{t}L^{r}_{x}}&\leq \|P_{\le N}
   e^{-itH}\phi_{k}\|^{(r-2)/r}_{L^{\infty}_{t}L^{2}_{x}}\|P_{\le N}
 e^{-itH}\phi_{k}\|^{2/r}_{L^{\infty}_{t}L^{\infty}_{x}}\\
&\leq \|\phi_{k}\|^{(r-2)/r}_{L^{2}}\|P_{\le N} e^{-itH}\phi_{k}\|^{2/r}_{L^{\infty}_{t}L^{\infty}_{x}}\\
&\leq \Lambda^{(r-2)/r}\|P_{\le N} e^{-itH}\phi_{k}\|^{2/r}_{L^{\infty}_{t}L^{\infty}_{x}}.
\end{align*}
Thus by \eqref{Pf1} we obtain, for $k$ large enough,
\begin{equation}\label{Cu1}
\|P_{\le N} e^{-itH}\phi_{k}\|_{L^{\infty}_{t}L^{\infty}_{x}}\geq \left(\frac{A_{1}}{4}\right)^{r/2}\Lambda^{1-r/2}.
\end{equation}
In view of Lemmas~3.1 and 3.2 from \cite{PoRoTh2015}, there exists
$c>0$ independent of $\phi_k$ and $t$ such that for all $x\in\R^d$,
\begin{equation*}
  |P_{\le N} e^{-itH}\phi_{k}(x)|\lesssim N^{n/2} e^{-c|y|^2/N^2}
  \(\int_{\R^n}\left| P_{\le N} e^{-itH}\phi_{k}(y,z)\right|^2dy\)^{1/2}.
\end{equation*}
Since $P_{\le N}$ localizes the frequencies in $z$, Bernstein
inequality implies
\begin{equation*}
  \int_{\R^n}\left| P_{\le N} e^{-itH}\phi_{k}(y,z)\right|^2dy\lesssim
  N^{d-n} \int_{\R^d}\left| P_{\le N} e^{-itH}\phi_{k}(y,z)\right|^2dydz,
\end{equation*}
and so
\begin{equation*}
  |P_{\le N} e^{-itH}\phi_{k}(x)| \lesssim N^{d/2} e^{-c|y|^2/N^2}\Lambda.
\end{equation*}
  We deduce from \eqref{Cu1} that for $R$ sufficiently large,
\begin{equation}\label{Cu12}
\|P_{\le N} e^{-itH}\phi_{k}\|_{L^{\infty}_{t}L^{\infty}_{|y|\leq R}}\geq \frac{1}{2\Lambda^{r/2-1}}\left(\frac{A_{1}}{4}\right)^{r/2}.
\end{equation}
It follows that there exist $t^{1}_{k}\in \R$, $z^{1}_{k}\in
\mathbb{R}^{d-n}$ and $y^{1}_{k}\in \mathbb{R}^{n}$, $|y^{1}_{k}|\leq R$, such that
\begin{equation}\label{Cas}
|P_{\le N} e^{-it^{1}_{k}H}\phi_{k}|(y^{1}_{k},z^{1}_{k})\geq  \frac{1}{4\Lambda^{r/2-1}}\left(\frac{A_{1}}{4}\right)^{r/2}.
\end{equation}
Since $|y^{1}_{k}|\leq R$, possibly after extracting a subsequence, we
get $y^{1}_{k}\rightarrow y^{1}$. Let 
\begin{equation*}
w_{k}(x)=e^{-it^{1}_{k}H}\phi_{k}(y, z+ z^{1}_{k}).
\end{equation*}
Then $\left\{w_{k}\right\}^{\infty}_{k=1}$ is uniformly bounded in
$B_{1}$ and there exists $\psi^{1}\in B_{1}$  such that, passing to a
subsequence if necessary,  $w_{k}\rightharpoonup \psi^{1}$ in $B_1$ as $k\rightarrow \infty$. In particular, $\|\psi^{1}\|_{B_{1}}\leq \Lambda$. As
$|P_{\le N} e^{-it^{1}_{k}H}\phi_{k}|(y^{1},z^{1}_{k})=|P_{\le N} w_{k}|(y^{1},0)$, by \eqref{Cas} we get
\begin{equation*}
|P_{\le N} \psi^{1}|(y^{1},0)\geq  \frac{1}{4\Lambda^{r/2-1}}\left(\frac{A_{1}}{4}\right)^{r/2}.
\end{equation*}
We note that the previous computations yield
\begin{align*}
  \|\psi^1\|_{L^2(\R^d)} \ge \|P_{\le N}\psi^1\|_{L^2(\R^d)} \gtrsim
  |P_{\le N} \psi^{1}|(y^{1},0)&\gtrsim \frac{1}{N^{d/2}}
    \frac{A_1^{r/2}}{\Lambda^{r/2-1}}\\&  \ge C_1
    \(\frac{A_1}{\Lambda}\)^{\tfrac{d}{2(1-s)}} \frac{A_1^{\si+1}}{\Lambda^{\si}},
  \end{align*}
  for a universal constant $C_1$. 
Set $W^{1}_{k}(x):=\phi_{k}(x)-e^{it^{1}_{k}H}\psi^{1}(y,
z-z^{1}_{k})$: $W^{1}_{k} \rightharpoonup 0$ in $B_{1}$. Furthermore, since
\begin{equation*}
\|\psi^{1}\|^{2}_{\dot{B}^{1}}=\lim_{k\rightarrow\infty}\left\langle \psi^{1}, e^{-it_k^1H}\phi_{k}(\cdot,\cdot+z^{1}_{k})\right\rangle=\lim_{k\rightarrow\infty}\left\langle e^{-it_k^1H}\psi^{1}, \phi_{k}(\cdot,\cdot+z^{1}_{k})\right\rangle,
\end{equation*}
this implies that
\begin{align*}
\|\phi_{k}\|^{2}_{\dot{B}_{1}}&=\|\psi^{1}\|^{2}_{\dot{B}_{1}}+\|W^{1}_{k}\|^{2}_{\dot{B}_{1}}+o_{k}(1),\\
\|\phi_{k}\|^{2}_{L^{2}}&=\|\psi^{1}\|^{2}_{L^{2}}+\|W^{1}_{k}\|^{2}_{L^{2}}+o_{k}(1),
\end{align*}
as $k\rightarrow\infty$. Thus \eqref{El23} and \eqref{El4} hold. In particular we see that $\|W^{1}_{k}\|^{2}_{B_{1}}\leq \Lambda$. 

We next replace $\left\{\phi_{k}\right\}^{\infty}_{k=1}$  by $\left\{W^{1}_{k}\right\}^{\infty}_{k=1}$ and repeat the same argument.\\
If
$A_{2}:=\limsup_{k\rightarrow\infty}\|e^{-itH}W^{1}_{k}\|_{L^{\infty}_{t}L^{r}_{x}}=0$,
we can take $\psi^{j}=0$ for every $j\geq 2$ and the proof is
over. Notice that the property \eqref{El2} is immediate consequence of
\eqref{Eh}. Otherwise there exist $\psi^{2}\in B_{1}$, a sequence of
time $\left\{t^{2}_{k}\right\}^{\infty}_{k=1}\subset \mathbb{R}$ and
sequence $\left\{z^{2}_{k}\right\}^{\infty}_{n=1}\subset
\mathbb{R}^{d-n}$ such that
$e^{-it^{2}_{k}H}W^{1}_{k}(\cdot,\cdot+z^{2}_{k})\rightharpoonup
\psi^{2}$ with 
\begin{equation*}
\|\psi^{2}\|_{L^2} \ge C_1
    \(\frac{A_2}{\Lambda}\)^{\tfrac{d}{2(1-s)}} \frac{A_2^{\si+1}}{\Lambda^{\si}}.
\end{equation*}
We now show that 
\begin{equation}\label{Orto}
|t^{2}_{k}-t^{1}_{k}|+|z^{2}_{k}-z^{1}_{k}|\Tend k \infty \infty.
\end{equation}
Let
$g_{k}:=e^{-it^{1}_{k}H}\phi_{k}(\cdot,\cdot+z^{1}_{k})-\psi^{1}=e^{-it^{1}_{k}H}
W_k^1 $. Notice that $g_{k}\rightharpoonup 0$ in $B_{1}$. Moreover, by
definition $e^{-i(t^{2}_{k}-t^{1}_{k})H}g_{k}(\cdot,\cdot+(z^{2}_{k}-z^{1}_{k}))\rightharpoonup \psi^{2}\neq 0$ weakly in $B_{1}$. Suppose by contradiction that $|t^{2}_{k}-t^{1}_{k}|+|z^{2}_{k}-z^{1}_{k}|$ is bounded. Then, after possible extraction, $t^{2}_{k}-t^{1}_{k}\rightarrow t^{\ast}$ and  $z^{2}_{k}-z^{1}_{k}\rightarrow z^{\ast}$. However, since $g_{k}\rightharpoonup 0$, we infer that  $e^{-i(t^{2}_{k}-t^{1}_{k})H}g_{k}(\cdot,\cdot+(z^{2}_{k}-z^{1}_{k}))\rightharpoonup 0$, which is impossible.

An argument of iteration and orthogonal extraction allows us to construct $\big\{t^{j}_{k}\big\}_{j\geq1}\subset \mathbb{R}$, $\big\{z^{j}_{k}\big\}_{j\geq1}\subset \mathbb{R}^{d-n}$ and the sequence of functions $\left\{\psi^{j}\right\}_{j\geq 1}$ in $B_{1}$ such that the properties \eqref{El1}, \eqref{El2} and \eqref{El23} hold and
\begin{equation*}
\|\psi^{M}\|_{L^2} \ge C_1
    \(\frac{A_M}{\Lambda}\)^{\tfrac{d}{2(1-s)}}
    \frac{A_M^{\si+1}}{\Lambda^{\si}}.
  \end{equation*}
  In view of \eqref{El23}, we obtain
\begin{equation*}
\frac{1}{\Lambda^{2\si+2+\tfrac{d}{1-s}}}\sum^{\infty}_{M=1}A^{2\si+\tfrac{d}{1-s} }_{M}\lesssim \Lambda^{2}, 
\end{equation*}
hence $A_{M}\rightarrow 0$ as $M\rightarrow \infty$. Finally, from \eqref{Eh} we infer that
\begin{equation*}
\|e^{-itH}W^{M}_{k}\|_{\ell^{p}_{\gamma}L^{q}L^{r}}\lesssim\Lambda^{1-\theta}A^{\theta}_{M},  
\end{equation*}
and the property \eqref{El2} holds.
\smallbreak

\noindent\textbf{Step 2.} It remains to show \eqref{Lps}. To this end,
we show that for all $M\ge 1$,
\begin{equation}\label{Iq}
\Big\|\sum^{M}_{j=1}e^{it^{j}_{k}H}\psi^{j}(\cdot, \cdot-z_{k})\Big\|^{2\sigma+2}_{L^{2\sigma+2}}=\sum^{M}_{j=1}\|e^{it^{j}_{k}H}\psi^{j}\|^{2\sigma+2}_{L^{2\sigma+2}}+o_{k}(1) .
\end{equation}
 We proceed  as in \cite[Lemma 2.3]{DuyHolmerRoude2008}.
By reordering, we can choose $M^{\ast}\leq M$ such that\\
(i) For $1\leq j\leq M^{\ast}$: the sequence $\big\{t^{j}_{k}\big\}_{k\geq1}$ is bounded.\\
(ii) For $M^{\ast}+1\leq j\leq M $: we have that $\lim_{k\rightarrow\infty} |t^{j}_{k}|=\infty $.\\
Consider the inequality 
\[ \left|\Big|\sum^{M}_{j=1}   z_{j}  \Big|^{2\sigma+2} -\sum^{M}_{j=1}|z_{j}|^{2\sigma+2} \right| \leq C_{\sigma,M}\sum_{j\neq j'}| z_{j} || z_{j} | ^{2\sigma+1}, \]
for $z_{j}\in \mathbb{C}$, $j=1$, $2$, $\ldots$, $M$. If $1\leq j\leq
\ell\leq M^{\ast}$, the pairwise orthogonality (in space) \eqref{El1}
leads the cross terms in the sum of the left side of \eqref{Iq} to
vanish as $k\rightarrow \infty$. Therefore, 
\begin{equation}\label{Iqx1}
\left\|\sum^{M^{\ast}}_{j=1}e^{it^{j}_{k}H}\psi^{j}(\cdot, \cdot-z_{k})\right\|^{2\sigma+2}_{L^{2\sigma+2}}=\sum^{M^{\ast}}_{j=1}\left\|e^{it^{j}_{k}H}\psi^{j}\right\|^{2\sigma+2}_{L^{2\sigma+2}}+o_{k}(1).
\end{equation}
On the other hand, if $ M^{\ast}+1\leq j\leq M$, then
$|t^{j}_{k}|\rightarrow +\infty$ and, from Lemma~\ref{lem:dispLp},
\begin{equation}\label{qes}
\lim_{k\rightarrow\infty}\left\|e^{it_k^jH}\psi^{j}\right\|^{r}_{L^{r}}=0.
\end{equation}
Moreover, since (see proof of Step 1)
\begin{equation}\label{scs}
\lim_{M\rightarrow\infty}\left(\lim_{k\rightarrow\infty}\|e^{-itH}W^{M}_{k}\|_{L^{\infty}_{t}L_{x}^{r}}\right)=0,
\end{equation}
combining \eqref{Iqx1}, \eqref{qes} and \eqref{scs}, we obtain \eqref{Iq}. This show the last statement of the proposition and the proof is complete.
\end{proof}


Finally, we will show the following result related with the linear profile decomposition.

\begin{lemma}\label{apx}
Let $M\in \mathbb{N}$ and let $\left\{\psi^{j}\right\}^{M}_{j=0}\subset B_{1}$ satisfy
\begin{equation*}
\sum^{M}_{j=0}S(\psi^{j})-\eps\leq S\left(\sum^{M}_{j=0}\psi^{j}\right)\leq\beta-\eta, \quad 
-\eps\leq I\left(\sum^{M}_{j=0}\psi^{j}\right)\leq \sum^{M}_{j=0}I(\psi^{j})+\eps.
\end{equation*} 
where $\eps>0$ and $2\eps<\eta$. Then for all $0\leq j\leq M$ we have $\psi^{j}\in \mathcal{K}^{+}$.
\end{lemma}
\begin{proof}
Assume by contradiction there exists $k\in \left\{0,1,\ldots, M\right\}$ such that $I(\psi^{k})<0$.  Using the definition of $(\psi^{k})_{\lambda}^{1,0}$ (see \eqref{Fm}) it is not difficult to show that there exists $\lambda<0$ such that $I((\psi^{k})_{\lambda}^{1,0})>0$. This implies that there exists $\lambda_{0}<0$ such that $I((\psi^{k})_{\lambda_{0}}^{1,0})=0$. Moreover, a simple calculation shows that $\partial_{\lambda}B^{1,0}((\psi^{k})_{\lambda}^{1,0})\geq 0$ where $B^{1,0}$ is given by \eqref{Co1}. Thus, by Lemma \ref{L1} we get
\begin{equation*}
B^{1,0}(\psi^{k})\geq B^{1,0}((\psi^{k})_{\lambda_{0}}^{1,0})=S((\psi^{k})_{\lambda_{0}}^{1,0})\geq\beta.
\end{equation*}
Notice that $B^{1,0}(\psi^{j})\geq 0$ for $0\leq j \leq M$, by Lemma \ref{L1}. Since $2\eps<\eta$, we obtain
\begin{align*}
\beta&\leq \sum^{M}_{j=0}B^{1,0}(\psi^{j})=\sum^{M}_{j=0}\left( S(\psi^{j})-\frac{1}{4}I(\psi^{j})\right)\\
&\leq S\left(\sum^{M}_{j=0}\varphi^{j}\right)+\eps-\frac{1}{4} I\left(\sum^{M}_{j=0}\varphi^{j}\right)+\frac{1}{4}\eps\leq \beta-\eta+2\eps<\beta,
\end{align*}
This is absurd. Therefore, we  infer that $I(\psi^{j})\geq 0$ for all $0\leq j\leq M$. In particular, $S(\psi^{j})=B^{1,0}(\psi^{j})+\frac{1}{2\sigma+2}I(\psi^{j})\geq 0$ and
\begin{equation*}
\sum^{M}_{j=0}S(\psi^{j})\leq S\left(\sum^{M}_{j=0}\psi^{j}\right)+\eps<\beta,
\end{equation*}
which implies that $S(\psi^{j})<\beta$. It follows (see Lemma~\ref{L3}) that $\psi^{j}\in \mathcal{K}^{+}$. This completes the proof.
\end{proof}


\subsection{Construction of a critical element}

We define the critical action level $\tau_{c}$ by
\begin{equation*}
\tau_{c}:=\sup\left\{\tau: S(\varphi)<\tau\,\, \text{and  $\varphi\in \mathcal{K}^{+}$ implies  $\|u\|_{\ell_{\gamma}^{p}L^{q}L^{r}}<\infty$}\right\}.
\end{equation*}
Here, $u(t)$ is the corresponding solution of \eqref{GP} with $u(0)=\varphi$.  We observe that $\tau_{c}$ is a strictly positive number. Indeed, if $\varphi\in \mathcal{K}^{+}$, by Lemmas \ref{L4} and \ref{axc}  we see that $\|e^{-itH}\varphi\|_{\ell_{\gamma}^{p}L^{q}L^{r}}\lesssim  \|\varphi\|_{B_{1}}\lesssim S(\varphi)$. Therefore, taking $\tau>0$ sufficiently small we obtain that  $\|u\|_{\ell_{\gamma}^{p}L^{q}L^{r}}<\infty$ by Lemma~\ref{ssp}. Hence $0<\tau_{c}\leq\beta$. We prove that $\tau_{c}=\beta$ by contradiction. 

We assume $\tau_{c}<\beta$.  By the definition of $\tau_{c}$, there exists a sequence of solutions $u_{k}$ to \eqref{GP} in $B_{1}$ with initial data $\phi_{k}\in \mathcal{K}^{+}$ such that $S(\phi_{k})\rightarrow \tau_{c}$ and $\|u_{k}\|_{\ell_{\gamma}^{p}L^{q}L^{r}}=\infty$.
In the next results, we construct a critical solution $u_{c}(t)\in
\mathcal{K}^{+}$ of \eqref{GP} such that $S(u_{c}(t))=\tau_{c}$ and
$\|u_{c}\|_{\ell_{\gamma}^{p}L^{q}L^{r}}=\infty$. Moreover, we prove that
there exists a continuous path $z(t)$ in $\mathbb{R}^{d-n}$ such that
the critical solution $u_{c}$ has the property that
$K=\left\{u_{c}(\cdot, \cdot-z(t))\right\}$ is precompact in
$B_{1}$. This is where the requirement $\si\ge \tfrac{1}{2}$ appears, in addition
to the previous assumption $\tfrac{2}{d-n}<\si<\tfrac{2}{d-2}$.

\begin{proposition}[Critical element]\label{P11ce}
  Let $n=1$ and  $\si\ge \tfrac{1}{2}$ with
  $\tfrac{2}{d-1}<\si<\tfrac{2}{d-2}$.
We assume that $\tau_{c}<\beta$. Then there exists $u_{c, 0}\in B_{1}$  such that the corresponding solution $u_{c}$ to \eqref{GP} with initial data $u_{c}(0)=u_{c, 0}$ satisfies $ u_{c}(t)\in \mathcal{K}^{+}$, $S(u_{c}(t))=\tau_{c}$ and $\|u_{c}\|_{\ell_{\gamma}^{p}L^{q}L^{r}}=\infty$.
\end{proposition}
\begin{proof}
Since $S(\phi_{k})\rightarrow \tau_{c}$, from Lemma \ref{L4} we see
that $\left\{\phi_{k}\right\}^{\infty}_{k=1}$ is bounded in
$B_{1}$. Indeed, $\|\phi_{k}\|_{B_{1}}\lesssim
S(\phi_{k})$, and $S(\phi_{k})\leq\beta$. Thus, by Proposition
\ref{Ls3}, up to extracting to a 
subsequence, we get  
\begin{equation}\label{Exs}
\phi_{k}=\sum^{M}_{j=1}e^{it^{j}_{k}H}\psi^{j}(\cdot, \cdot-z_{k})+W^{M}_{k} \quad \text{for all $M\in\mathbb{N}$},
\end{equation}
and the sequence satisfies 
\begin{align*}
S(\phi_{k})&=\sum^{M}_{j=1}S\left( e^{it^{j}_{k}H}\psi^{j}\right)+S(W^{M}_{k})+o_{k}(1),\\    
I(\phi_{k})&=\sum^{M}_{j=1}I\left( e^{it^{j}_{k}H}\psi^{j}\right)+I(W^{M}_{k})+o_{k}(1).
\end{align*}
By using the fact that $\phi_{k}\in \mathcal{K}^{+}$, we infer that there exists $\eps$, $\eta>0$ such that $2\eps<\eta$ and 
\begin{align*}
S(\phi_{k})&\leq \beta-\eta,\\  
S(\phi_{k})&\geq \sum^{M}_{j=1}S\left( e^{it^{j}_{k}H}\psi^{j}\right)+S(W^{M}_{k})-\eps,\\
I(\phi_{k})&\geq-\eps,\\
I(\phi_{k})&\leq \sum^{M}_{j=1}I\left( e^{it^{j}_{k}H}\psi^{j}\right)+I(W^{M}_{k})+\eps
\end{align*}
for sufficiently large  $k$. Thus, from Lemma~\ref{apx} we obtain that 
\begin{equation}\label{impor}
e^{it^{j}_{k}H}\psi^{j}\in \mathcal{K}^{+},\quad W^{M}_{k}\in \mathcal{K}^{+} \quad \text{for sufficiently large $k$.}
\end{equation}
This implies that $S( e^{it^{j}_{k}H}\psi^{j})\geq 0$, $S(W^{M}_{k})\geq 0$ and for each $1\leq j \leq M$,
\begin{equation}\label{icd}
0\leq \limsup_{k\rightarrow\infty}S( e^{it^{j}_{k}H}\psi^{j})\leq \limsup_{k\rightarrow\infty}S(\phi_{k})=\tau_{c}. 
\end{equation}
Now we have two cases: (i)~$\limsup_{k\rightarrow\infty}S(
e^{it^{j}_{k}H}\psi^{j})=\tau_{c}$ fails for all $j$, or (ii)~equality
holds in \eqref{icd} for some $j$.\\
\textbf{Case (i)}:  In this case, for  each $1\leq j \leq M$ there exists $\eta_{j}>0$ such that
\begin{equation}\label{cc1}
\limsup_{k\rightarrow\infty}S( e^{it^{j}_{k}H}\psi^{j})\leq\tau_{c}-\eta_j, \quad S( e^{it^{j}_{k}H}\psi^{j})\geq 0, \quad I( e^{it^{j}_{k}H}\psi^{j})\geq 0.
\end{equation}                                                   
Suppose that $t^{j}_{k}\rightarrow t^{\ast}$. If $t^{\ast}<\infty$ for
some $j$ (at most one such $j$ exists by the orthogonality of the
parameters \eqref{El1}), then from the continuity of the linear flow
we infer that 
\begin{equation}\label{Lf}
e^{it^{j}_{k}H}\psi^{j}\Tend k \infty e^{it^{\ast}H}\psi^{j}\quad \text{strongly in $B_{1}$.}
\end{equation}
We set $\psi^{j}_{\ast}=\text{NLS}(t^{\ast})(e^{it^{\ast}H}\psi^{j})$,
where we recall that $\text{NLS}(t)\varphi$ denotes the solution to \eqref{GP} with
  initial datum $u_0=\varphi$. Notice that $\text{NLS}(-t^{\ast})\psi^{j}_{\ast}=e^{it^{\ast}H}\psi^{j}$. Moreover, by \eqref{impor} and \eqref{cc1} we have that $\psi^{j}_{\ast}\in \mathcal{K}^{+}$ and $S( \psi^{j}_{\ast})<\tau_{c}$. Thus, by definition of $\tau_{c}$ we get $\|\text{NLS}(\cdot)\psi^{j}_{\ast}\|_{\ell_{\gamma}^{p}L^{q}L^{r}}<\infty$. Finally, by \eqref{Lf} we obtain
\begin{equation}\label{Lf11}
\|\text{NLS}(-t^{j}_{k})\psi^{j}_{\ast}-e^{it^{\ast}H}\psi^{j}\|_{B_{1}}\rightarrow 0 \quad \text{as $k\rightarrow\infty$.}
\end{equation}
On the other hand, suppose that $|t^{j}_{k}|\to\infty$:
$\|e^{it^{j}_{k}H}\psi^{j}\|_{L^{2\sigma+2}}\to 0$, and therefore
\begin{equation}\label{Lf11si}
\lim_{k\rightarrow\infty}S\(e^{it^{j}_{k}H}\psi^{j}\)=\frac{1}{2}\|\psi^{j}\|^{2}_{B_{1}}<\tau_{c}<\beta.
\end{equation}
By Lemma \ref{Wape}, there exists $\psi^{j}_{\ast}$ such that $\psi^{j}_{\ast}\in \mathcal{K}^{+}$ and
\begin{equation}\label{Lsd}
\|\text{NLS}(-t^{j}_{k})\psi^{j}_{\ast}-e^{it^{j}_{k}H}\psi^{j}\|_{B_{1}}\Tend
k \infty 0.
\end{equation}
Moreover, by \eqref{Lf11si} we have $S(\psi_{\ast}^{j})=\frac{1}{2}\|\psi^{j}\|^{2}_{B_{1}}<\tau_{c}$. Again, by definition of $\tau_{c}$ we see that $\|\text{NLS}(\cdot)\psi^{j}_{\ast}\|_{\ell_{\gamma}^{p}L^{q}L^{r}}<\infty$.

In either case, we obtain a new profile  $\psi^{j}_{\ast}$ for the given $\psi^{j}$ such that \eqref{Lsd} holds and $\|\text{NLS}(\cdot)\psi^{j}_{\ast}\|_{\ell_{\gamma}^{p}L^{q}L^{r}}<\infty$. We rewrite $\phi_{k}$ as follows (see \eqref{Exs}):
\begin{equation*}
\phi_{k}=\sum^{M}_{j=1}\text{NLS}(-t^{j}_{k})\psi^{j}_{\ast}(\cdot,\cdot-z^{j}_{k})+\tilde{W}^{M}_{k},
\end{equation*}
where
\begin{equation}\label{Lxz}
\tilde{W}^{M}_{k}=\sum^{M}_{j=1}\left[e^{it^{j}_{k}H}\psi^{j}(\cdot,\cdot-z^{j}_{k})-\text{NLS}(-t^{j}_{k})\psi^{j}_{\ast}(\cdot,\cdot-z^{j}_{k})\right]+{W}^{M}_{k}.
\end{equation}
We observe that by Lemma~\ref{axc},
\begin{gather*}
\|e^{-itH}\tilde{W}^{M}_{k}\|_{\ell_{\gamma}^{p}L^{q}L^{r}}\lesssim\sum^{M}_{j=1}\| e^{-it_k^jH}\psi^{j}-\text{NLS}(-t^{j}_{k})\psi^{j}_{\ast}\|_{B_{1}}+\|e^{-itH}{W}^{M}_{k}\|_{\ell_{\gamma}^{p}L^{q}L^{r}}.
\end{gather*}
Thus, we have
\begin{equation}\label{LLi}
\lim_{M\rightarrow \infty}\(\lim_{k\rightarrow \infty}\|e^{-itH}\tilde{W}^{M}_{k}\|_{\ell_{\gamma}^{p}L^{q}L^{r}}\)=0.
\end{equation}
The idea now is to approximate
\begin{equation*}
\text{NLS}(t)\phi_{k}\approx\sum^{M}_{j=1}\text{NLS}(t-t^{j}_{k})\psi^{j}_{\ast}(\cdot,\cdot-z^{j}_{k}),
\end{equation*}
and use the  approximation theory from Lemma~\ref{lpt} to obtain $\|\text{NLS}(\cdot)\phi_{k}\|_{\ell_{\gamma}^{p}L^{q}L^{r}}<\infty$, which is a contradiction. With this in mind, we define
\begin{equation*}
u_{k}(t)=\text{NLS}(t)\phi_{k}, \quad v^{j}_{k}(t)=\text{NLS}(t-t^{j}_{k})\psi^{j}_{\ast}(\cdot,\cdot-z^{j}_{k}), \quad {u}^{M}_{k}(t)=\sum^{M}_{j=1}v_{k}^{j}(t).
\end{equation*}
A simple calculation shows that $i\partial_{t}{u}^{M}_{k}-H{u}_{k}^{M}+|{u}^{M}_{k}|^{2\sigma}{u}^{M}_{k}=e_{k}^{M}$, where
\begin{equation*}
e_{k}^{M}=|{u}^{M}_{k}|^{2\sigma}{u}^{M}_{k}-\sum^{M}_{j=1}|v_{k}^{j}|^{2\sigma}v_{k}^{j}.
\end{equation*}
and
\begin{equation}\label{dp}
u_{k}(0)-{u}_{k}^{M}(0)=\tilde{W}^{M}_{k}.
\end{equation}
We rely on the following two claims.

\textbf{Claim 1.} There exists $A>0$ (independent of $M$) such that for each $M$, there exists $k_{1}=k_{1}(M)$ with the following property: if $k>k_{1}$ then we have the following estimate
\begin{equation}\label{cla123}
\|u_{k}^{M}\|_{\ell_{\gamma}^{p}L^{q}L^{r}}\leq A.
\end{equation}

\textbf{Claim 2.} 
There exists $k_{2}=k_{2}(M, \eps(A))$ such that if $k>k_{2}$, then we have the following estimate
\begin{equation}\label{cla1}
\|e_{k}^{M}\|_{\ell_{\gamma}^{\tilde{p}'}L^{\tilde{q}'}L^{r'}}\leq \eps(A),
\end{equation}
where $A$ is given by \eqref{cla123} and $\eps(A)$ is the
associate value provided by Lemma~\ref{lpt}.
\smallbreak

To prove Claim~1, we note that following the same strategy as in
e.g. \cite{Keraani01,KenigMerle2006,HolmerRoudenko2008,FangXieCaze2011},
relying on an interpolation of the norm involved in the asymptotic
smallness of $W_k^M$ (\eqref{El2}, in our case) by norms of the form
$L^\gamma_{t,x}$ and $L^\infty H^1$, seems doomed. Indeed, since
$q>p$, it does not seem easy to control the $\ell_\gamma^pL^qL^r$ in
the fashion. However, as noticed in \cite{BanicaVisciglia2016}, it is
possible to do without, by just using the fact that the Lebesgue
exponents at stake are all finite. We therefore resume the main ideas
from \cite[Appendix~A]{BanicaVisciglia2016}, to obtain
\begin{equation}
  \label{eq:BV16-app}
  \limsup_{k \to \infty} \|u_{k}^{M}\|_{\ell_{\gamma}^{p}L^{q}L^{r}}^{2\si+1} \le
  2\sum_{j=1}^M \|
  \text{NLS}(\cdot)\psi_\ast^j\|_{\ell_{\gamma}^{p}L^{q}L^{r}}^{2\si+1}. 
\end{equation}
Recall the identities $\tilde{p}'=(2\sigma+1)p$, $\tilde{q}'=(2\sigma+1)q$ and $r'=(2\sigma+1)r$.
To prove \eqref{eq:BV16-app}, we first notice that if $f_1,f_2\in
C(\R;B_1)\cap \ell_{\gamma}^{p}L^{q}L^{r}$ and
\begin{equation*}
  |t_k-s_k|+|z_k-\eta_k|\Tend  k \infty \infty,
\end{equation*}
then
\begin{equation}\label{eq:propA1}
  \left\|
    |f_1(t-t_k,y,z-z_k)|^{2\si}f_2(t-s_k,y,z-\zeta_k)\right\|_{\ell_{\gamma}^{\tilde
      p'}L^{\tilde q'}L^{r'}}\Tend k\infty 0.
\end{equation}
Indeed, H\"older inequality in space yields
\begin{align*}
  \Big\|
    |f_1(t-t_k,y,z-z_k)|^{2\si}f_2(t-s_k,y,& z-\zeta_k)\Big\|_{\ell_{\gamma}^{\tilde
  p'}L^{\tilde q'}L^{r'}}\\
  &\le \left\|
    \|f_1(t-t_k)\|_{L^r}^{2\si}\|f_2(t-s_k)\|_{L^r}\right\|_{\ell_{\gamma}^{\tilde
      p'}L^{\tilde q'}},
\end{align*}
 and \eqref{eq:propA1} follows in the case $|t_k-s_k|\Tend  k \infty
 \infty$, since $\tilde p'$ and $\tilde q'$ are finite. In the case
 where this sequence is bounded, for $\gamma_0\ge 1$, H\"older
 inequality in space and time yields
 \begin{align*}
   \Big\|
    |f_1(t,y,z-z_k)|^{2\si}&f_2(t+t_k-s_k,y,
     z-\zeta_k)\Big\|_{\ell_{|\gamma|\ge \gamma_0}^{\tilde
  p'}L^{\tilde q'}L^{r'}}\\
   &\le 
    \|f_1\|_{\ell_{|\gamma|\ge \gamma_0}^{
  p}L^{ q}L^{r}}^{2\si}\|f_2(t+t_k-s_k)\|_{\ell_{|\gamma|\ge \gamma_0}^{
  p}L^{ q}L^{r}}\Tend {\gamma_0} \infty 0.
 \end{align*}
 Now for $t$ fixed,
 \begin{align*}
   \Big\|
    |f_1(t,y,z-z_k)|^{2\si}&f_2(t+t_k-s_k,y,
     z-\zeta_k)\Big\|_{L_x^{r'}}\\
   &=\Big\|
    |f_1(t,y,z)|^{2\si}f_2(t+t_k-s_k,y,
     z+z_k-\zeta_k)\Big\|_{L_x^{r'}}\Tend k\infty 0 ,
  \end{align*}
 since $|z_k-\zeta_k|\to \infty$,  $|f_1(t,\cdot)|^{2\si}\in L^{\tfrac{r}{2\si}}$ for \emph{all}
 $t$, using the property 
 $f_1\in C_t H^1$ and Sobolev embedding, and, for the same reason,
 \begin{equation*}
   \left\{ f_2(t+t_k-s_k,y,
     z+z_k-\zeta_k),\ k\in \N\right\} \text{ is compact in }L^r,\quad
   \forall t. 
 \end{equation*}
 Invoking H\"older inequality in space again,
 \begin{align*}
  \Big\|
    |f_1(t,y,z)|^{2\si}f_2(t+t_k-s_k,y,& z+z_k-\zeta_k)\Big\|_{L^{r'}}\\
  &\le 
    \|f_1(t)\|_{L^r}^{2\si}\|f_2(t+t_k-s_k)\|_{L^r},
 \end{align*}
 the Lebesgue dominated convergence theorem implies, for any given
 $\gamma_0\ge 1$,
 \begin{align*}
   \left\|
    |f_1(t,y,z-z_k)|^{2\si}f_2(t+t_k-s_k,y,
     z-\zeta_k)\right\|_{\ell_{|\gamma|\le \gamma_0}^{\tilde
  p'}L^{\tilde q'}L^{r'}}\Tend k \infty 0,
 \end{align*}
 hence \eqref{eq:propA1}.
 Now we observe that for $M\geq2$, there exists a constant $C_{M}>0$ such that
\begin{equation}\label{eq:orth}
\Bigg|
  \Big|\sum^{M}_{j=1}z_{j}\Big|^{2\sigma}\sum^{M}_{j=1}z_{j}-\sum^{M}_{j=1}|z_{j}|^{2\sigma}z_{j}\Bigg|\leq
C_{M}\sum_{1\leq j\neq \ell\leq M}|z_{j}|^{2\sigma}|z_{\ell}|.   
\end{equation}
Writing
\begin{align*}
 \|u_{k}^{M}\|_{\ell_{\gamma}^{p}L^{q}L^{r}}^{2\si+1} &= \Big\| \sum_{j=1}^M
 \text{NLS}(t-t_k^j)\psi_\ast^j(\cdot,\cdot-z_k^j)\Big\|_{\ell_{\gamma}^{p}L^{q}L^{r}}^{2\si+1} 
  \\ 
& \le  \Big\| \Big(\sum_{j=1}^M
        \big|\text{NLS}(t-t_k^j)\psi_\ast^j(\cdot,\cdot-z_k^j)\big|\Big)^{2\si+1}
        \Big\|_{\ell_{\gamma}^{\tilde p'}L^{\tilde
        q'}L^{r'}}\\
& \le  \Big\| \sum_{j=1}^M
 \Big| \text{NLS}(t-t_k^j)\psi_\ast^j(\cdot,\cdot-z_k^j )\Big|^{2\si+1}
 \Big\|_{\ell_{\gamma}^{\tilde p'}L^{\tilde  q'}L^{r'}}\\
   + \Big\| \Big(\sum_{j=1}^M
       \big| \text{NLS}(t-t_k^j)\psi_\ast^j(\cdot,\cdot-z_k^j&)\big|\Big)^{2\si+1}
    -\sum_{j=1}^M
 \Big| \text{NLS}(t-t_k^j)\psi_\ast^j(\cdot,\cdot-z_k^j)\Big|^{2\si+1}
    \Big\|_{\ell_{\gamma}^{\tilde p'}L^{\tilde
       q'}L^{r'}}.
\end{align*}
The last term goes to zero as $k\to \infty$, from \eqref{eq:propA1}
and \eqref{eq:orth}, hence \eqref{eq:BV16-app}
thanks to triangle inequality. Now using \eqref{El4} and \eqref{Lf11},
there exists $M_0$ such that
\begin{equation*}
  \|\psi_\ast^j\|_{B_1}\le \nu,\quad \forall j\ge M_0,
\end{equation*}
where $\nu$ is given by Lemma~\ref{ssp}. Lemma~\ref{ssp} then implies,
for all $j\ge M_0$,
\begin{equation*}
  \|
  \text{NLS}(\cdot)\psi_\ast^j\|_{\ell_{\gamma}^{p}L^{q}L^{r}}\le 2
  \|e^{-itH}\psi_\ast^j\|_{\ell_{\gamma}^{p}L^{q}L^{r}} \lesssim \|\psi_\ast^j\|_{B_1},
\end{equation*}
where we have used Lemma~\ref{axc}. For $\si\ge \tfrac{1}{2}$, we infer
\begin{equation*}
  \sum_{j=M_0}^\infty
  \|\text{NLS}(\cdot)\psi_\ast^j\|_{\ell_{\gamma}^{p}L^{q}L^{r}}^{2\si+1}
  \lesssim \sum_{j=M_0}^\infty
  \|\psi_\ast^j\|_{B_1}^{2\si+1}\lesssim
  \sum_{j=M_0}^\infty 
  \|\psi_\ast^j\|_{B_1}^{2}<\infty.
\end{equation*}
Now for $j<M_0$, we have seen  that
\begin{equation*}
   \|
  \text{NLS}(\cdot)\psi_\ast^j\|_{\ell_{\gamma}^{p}L^{q}L^{r}}<\infty,
\end{equation*}
hence Claim~1. 
 Claim~2 then follows from  \eqref{eq:propA1} and \eqref{eq:orth}. 

Next notice that combining \eqref{dp} and \eqref{LLi} we infer that for $\eps(A)$ there exists $M_{1}=M_{1}(\eps)$ such that for any $M>M_{1}$, then there exists $k_{3}=k_{3}(M_{1})$ such that if $k>k_{3}$
then we obtain
\begin{equation}\label{cla12}
\|e^{-itH}(u_{k}(0)-{u}_{k}^{M}(0))\|_{\ell_{\gamma}^{p}L^{q}L^{r}}\leq \eps(A).
\end{equation}
Therefore, by \eqref{cla123}, \eqref{cla1} and \eqref{cla12} we see that for $k\geq \max\left\{k_{1}, k_{2}, k_{3}\right\}$ we obtain that
$\|u_{k}^{M}\|_{\ell_{\gamma}^{p}L^{q}L^{r}}\leq A$,
$\|e_{k}^{M}\|_{\ell_{\gamma}^{p'}L^{q'}L^{r'}}\leq \eps(A)$ and
$\|e^{-itH}(u_{k}(0)-{u}_{k}^{M}(0))\|_{\ell_{\gamma}^{p}L^{q}L^{r}}\leq
\eps(A)$. Thus by Lemma~\ref{lpt} we get
$\|\phi_{k}\|_{\ell_{\gamma}^{p}L^{q}L^{r}}<\infty$, which is
absurd.\\ 
\textbf{Case (ii)}: We note that if equality  holds in \eqref{icd} for
some $j$ (we may assume $j=1$ by reordering), then $M=1$. In particular, $\limsup_{k\rightarrow\infty}S(W^{1}_{k})=0$. Since $S(W^{1}_{k})\sim \|W^{1}_{k}\|^{2}_{B_{1}}$ (see Lemma \ref{L4}), we have that $W^{1}_{k}\rightarrow0$ in $B_{1}$.
Thus  $\left\{\phi_{k}\right\}^{\infty}_{k=1}$ has only one nonlinear profile 
\begin{equation}\label{Pre}
\phi_{k}=e^{it^{1}_{k}H}\psi^{1}(\cdot, \cdot-z_{k})+W^{1}_{k}\quad \text{and $W^{1}_{k}\rightarrow0$ in $B_{1}$.}
\end{equation}
Suppose that $t^{1}_{k}\rightarrow t^{\ast}$. If $|t^{\ast}|<\infty$ (we
may then assume $t^{\ast}=0$),  we put $\psi^{\ast}=\psi^{1}$. Then as $k\rightarrow \infty$, $\|e^{it^{1}_{k}H}\psi^{1}-\text{NLS}(-t^{1}_{k})\psi^{\ast}\|_{B_{1}}\rightarrow 0$. Now if $|t^{\ast}|=\infty$, then $\|e^{it^{1}_{k}H}\psi^{1}\|_{L^{2\sigma+2}}\rightarrow 0$. This implies that 
\begin{equation*}
\frac{1}{2}\|\psi^{1}\|^{2}_{{B}^{1}} =\frac{1}{2}\| e^{it^{1}_{k}H}\psi^{1}\|^{2}_{{B}^{1}} =\lim_{k\rightarrow \infty}S\left( e^{it^{1}_{k}H}\psi^{1}\right)=\tau_{c}<\beta.
\end{equation*}
Thus, by Lemma \ref{Wape} there exists $\psi^{\ast}$  such that the
corresponding solution $\text{NLS}(t)\psi^{\ast}\in \mathcal K^{+}$ for all $t\in \mathbb{R}$ and
\begin{equation*}
\|e^{it^{1}_{k}H}\psi^{1}-\text{NLS}(-t^{1}_{k})\psi^{\ast}\|_{B_{1}}\rightarrow 0\quad \text{as $k\rightarrow\infty$.}
\end{equation*}
In either case, we set $u_{c,0}:=\psi^{\ast}$. We note that
$u_{c,0}\in \mathcal K^{+}$ and $S(u_{c,0})=S(\psi^{\ast})=\tau_{c}$. By \eqref{Pre} we can rewrite $\phi_{k}$ as
\begin{equation*}
\phi_{k}=\text{NLS}(-t^{1}_{k})\psi^{\ast}+\tilde{W}^{1}_{k},
\end{equation*}
where $\tilde{W}^{1}_{k}={W}^{1}_{k}+e^{it^{1}_{k}H}\psi^{1}-\text{NLS}(-t^{1}_{k})\psi^{\ast}$. Since $W^{1}_{k}\rightarrow0$ in $B_{1}$, it follows by Lemma \ref{axc}
\begin{equation*}
\lim_{M\rightarrow \infty}\left\{\lim_{k\rightarrow \infty}\|e^{-itH}\tilde{{W}}^{M}_{k}\|_{\ell_{\gamma}^{p}L^{q}L^{r}}\right\}=0.
\end{equation*}
Therefore, by the same argument as above (Case (i)) we infer that $\|u_{c}\|_{\ell_{\gamma}^{p}L^{q}L^{r}}=\infty$, which proves the proposition.
\end{proof}

\subsection{Extinction of the critical element}
In this subsection, we assume that $\|u\|_{\ell^{{p}}_{\gamma\geq 1}L^{q}L^{r}}=\infty$; we call it a forward critical element. 
We remark that the same argument as below does work in the case $\|u\|_{\ell^{{p}}_{\gamma\leq 1}L^{q}L^{r}}=\infty$.

\begin{lemma}\label{final} 
Let $u_{c}$ be the critical element given in Proposition \ref{P11ce}. Then $u_{c}=0$.
\end{lemma}

To prove Lemma \ref{final}, we need the following result.
\begin{lemma}\label{compass}
Let $u_{c}$  be the critical element given in Proposition \ref{P11ce}. Then there exists a function $z\in C([0, \infty);\mathbb{R}^{d-n})$ such that $\left\{u_{c}(t,\cdot,\cdot-z(t)); t\geq 0\right\}$ is relatively compact in $B_{1}$. In particular, we have the uniform localization of $u_{c}$:
\begin{equation}\label{Als}
\sup_{t\geq 0}\int_{|z+z(t)|>R}\left[|\nabla
  u(t,x)|^{2}+|u(t,x)|^{2\sigma+2}+|u(t,x)|^{2}\right]dx\Tend R \infty
0.
\end{equation}
\end{lemma}
\begin{proof}
By \cite[Appendix A]{DuyHolmerRoude2008} (see also proof of Proposition 6.1 in \cite{FangXieCaze2011}),  it is enough to show that the following condition is satisfied:
for every sequence $\left\{t_{k}\right\}^{\infty}_{k=1}$, $t_{k}\rightarrow\infty$, extracting a subsequence from $\left\{t_{k}\right\}^{\infty}_{k=1}$ if necessary, there exists $\left\{z_{k}\right\}^{\infty}_{k=1}\subset \mathbb{R}^{d-n}$ and  $\varphi\in B_{1}$ such that $u_{c}(t_{k}, \cdot,\cdot-z_{k})\rightarrow \varphi$ in $B_{1}$.\\
We set $\phi_{k}:=u_{c}(t_{k})$.  We note that $\phi_{k}$ satisfies:
\begin{equation}\label{zxcv}
S(\phi_{k})=\tau_{c}\quad \text{and} \quad \phi_{k}\in \mathcal{K}^{+}.
\end{equation}
Since $\|\phi_{k}\|^{2}_{B_{1}} \lesssim S(\phi_{k})$, it follows that
$\left\{\phi_{k}\right\}^{\infty}_{k=1}$ is bounded in $B_{1}$. Thus,
using the same argument developed in the proof of
Proposition~\ref{P11ce}, we obtain that
$\left\{\phi_{k}\right\}^{\infty}_{k=1}$ has only one 
nonlinear profile
\begin{equation*}
\phi_{k}=e^{it^{1}_{k}H}\psi^{\ast}(\cdot, \cdot-z_{k})+W^{1}_{k},
\end{equation*}
with $W^{1}_{k}\rightarrow0$ in $B_{1}$  (see proof of Case (ii) above). Assume that $|t^{1}_{k}|\rightarrow \infty$. Then we have two cases to consider. We first assume that $t^{1}_{k}\rightarrow -\infty$. By Lemma~\ref{axc} we see that
\begin{equation*}
\|e^{-itH}{u_{c}(t_{k})}\|_{\ell^{{p}}_{\gamma\geq 1}L^{q}L^{r}}\lesssim\|e^{-i(t-t^{1}_{k})H}{\psi^{\ast}}\|_{\ell_{\gamma\geq 1}^{{p}}L^{q}L^{r}}+\|{W^{1}_{k}}\|_{B_{1}}.
\end{equation*}
Since $W^{1}_{k}\rightarrow0$ in $B_{1}$ and 
\begin{equation*}
\lim_{k\rightarrow \infty}\|e^{-i(t-t^{1}_{k})H}\psi^{\ast}\|_{\ell_{\gamma\geq1}^{{p}}L^{q}L^{r}}=\lim_{k\rightarrow \infty}\|e^{-itH}\psi^{\ast}\|_{\ell_{\gamma\gtrsim-t^{1}_{k}}^{{p}}L^{q}L^{r}}=0,
\end{equation*}
it follows that $\|e^{-itH}{u_{c}(t_{k})}\|_{\ell_{\gamma\geq1}^{{p}}L^{q}L^{r}}\rightarrow 0$ as $k\rightarrow\infty$. In particular, for $k$ large, we have $\|e^{-itH}{u_{c}(t_{k})}\|_{\ell_{\gamma\geq1}^{{p}}L^{q}L^{r}}\leq \delta$, where $\delta$ is given in Lemma~\ref{ssp}. 
Then from Lemma~\ref{ssp} we obtain that
\begin{equation*}
\|\text{NLS}(t){u_{c}(t_{k})}\|_{\ell_{\gamma\geq1}^{{p}}L^{q}L^{r}} \lesssim  \delta,
\end{equation*}
 which is a absurd. Next, if $t^{1}_{k}\rightarrow \infty$, then a
 similar argument shows that 
\begin{equation*}
\|e^{-itH}{u_{c}(t_{k})}\|_{\ell_{\gamma\leq 1}^{{p}}L^{q}L^{r}}\leq \delta, \quad \text{for $k$ large.}
\end{equation*}
Again from Lemma \ref{ssp} we have $\|u_{c}\|_{\ell_{\gamma \lesssim
    t_{k} }^{{p}}L^{q}L^{r}}\lesssim \delta$. Since $t_{k}\rightarrow
\infty$ we infer that
$\|u_{c}\|_{\ell_{\gamma}^{{p}}L^{q}L^{r}}\lesssim \delta$, which is
also absurd. Therefore $t^{1}_{k}\rightarrow t^{\ast}$, $t^{\ast}\in
\mathbb{R}$. Thus
\[u_{c}(t_{k},\cdot,\cdot+z_k)\rightarrow e^{it^{\ast}H}\psi^{\ast}
  \text{ in }B_{1},\] and this completes the proof.

\end{proof}

\begin{proof}[Proof of Lemma~\ref{final}]
We proceed by a contradiction argument. Assume that $\varphi:=u_{c,0}\neq 0$. We observe that $G(\varphi)=0$ ($G$, we recall, is defined in \eqref{Moment}).  Indeed, suppose that $G(\varphi)\neq0$.  We define 
\begin{equation*}
\psi(x):=e^{iz\cdot z_{0}}\varphi(y,z), \quad\text{where}\quad 
z_{0}=- \frac{G(\varphi)}{\|\varphi\|^{2}_{L^{2}}}.
\end{equation*}
It is not difficult to show that $G(\psi)=0$,
$\|\nabla_{x}\psi\|^{2}_{L^{2}}< \|\nabla_{x}\varphi\|^{2}_{L^{2}}$
and  $\|\psi\|_{L^{2\sigma+2}}=\|\varphi\|_{L^{2\sigma+2}}$. Notice
that $\psi\in \mathcal K^{+}$. Indeed, since $\varphi\in \mathcal K^{+}$ we see that $S(\psi)< S(\varphi)=\tau_{c}<\beta$. Moreover, $I(\psi)\geq 0$. Assume by contradiction that $I(\psi)<0$. Then there exists $\lambda\in (0,1)$ such that $I(\lambda\psi)=0$. By using the fact $S(\varphi)\geq \frac{\sigma}{\sigma+1}\|\varphi\|^{2\sigma+2}_{L^{2\sigma+2}}$ we have
\begin{equation*}
S(\lambda\psi)=\frac{1}{2}I(\lambda\psi)+\frac{\sigma}{\sigma+1}\|\lambda\psi\|^{2\sigma+2}_{L^{2\sigma+2}}<\frac{\sigma}{\sigma+1}\|\psi\|^{2\sigma+2}_{L^{2\sigma+2}}=\frac{\sigma}{\sigma+1}\|\varphi\|^{2\sigma+2}_{L^{2\sigma+2}}<\beta,
\end{equation*}
which is absurd by Lemma~\ref{L1}. Therefore, $I(\psi)\geq 0$, $S(\psi)<\tau_{c}$ and $\psi\in \mathcal K^{+}$ (see Lemma~\ref{L3}). The corresponding solution $v\in C([0, \infty); B_{1})$ of \eqref{GP} with $v(0)=\psi$ is given by 
\begin{equation*}
v(t,y,z)=e^{i(z\cdot z_{0}-t|z_{0}|^{2})}u(t,y,z-2tz_{0}).
\end{equation*}
Since $\|u_{c}\|_{\ell_{\gamma}^{p}L^{q}L^{r}}=\infty$, it follows
that $\|v\|_{\ell_{\gamma}^{p}L^{q}L^{r}}=\infty$,  which is a
contradiction with the definition of $\tau_{c}$.\\
\textbf{Step 1.} We claim that
\begin{equation}\label{Fsa}
\lim_{t\rightarrow \infty}\frac{|z(t)|}{t}=0,
\end{equation}
where $z(t)$ is given in Lemma \ref{compass}.  The proof in \cite[Lemma 5.1]{DuyHolmerRoude2008} can be easily adapted to our case by considering the truncated center of mass of the form
\begin{equation*}
\Gamma_{R}(t)=\int_{\mathbb{R}^{d}}\phi_{R}(z)|u_{c}(t,x)|^{2}dx,
\end{equation*}
where $\phi_{R}(z)=R\phi(\tfrac{z}{R})$, $\phi(z)=(\theta(z_{1}),\theta(z_{2}), \ldots, \theta(z_{d-n}))$, $z\in \mathbb{R}^{d-n}$ such that $\theta\in C^{\infty}_{c}(\mathbb{R})$, $\theta(s)=1$ for $-1\leq s\leq 1$,
$\theta(s)=0$ for $|s|\geq 2^{1/3}$, $|\theta(s)|\leq |s|$,
$\|\theta\|_{L^{\infty}}\leq 2$ and
$\|\theta^{\prime}\|_{L^{\infty}}\leq 4$. 
Assume that \eqref{Fsa} is false. Then there exist a sequence $t_{k}\rightarrow \infty$ and $\alpha>0$ such that $|z(t_{k})|\geq \alpha t_{k}$.
Without loss of generality we may assume $z(0)=0$. For $R>0$ we set
\begin{equation*}
t_{0}(R)=\inf\left\{t\geq 0; |z(t)|\geq R\right\}.
\end{equation*}
We define $R_{k}=|z(t_{k})|$. Notice that $R_{k}\geq \alpha
t_{0}(R_{k})$ and $t_{0}(R_{k})\rightarrow\infty$ as
$k\rightarrow\infty$. On the other hand,
$\Gamma^{\prime}_{R}(t)=([\Gamma^{\prime}_{R}(t)]_{1},
[\Gamma^{\prime}_{R}(t)]_{2}\ldots, [\Gamma^{\prime}_{R}(t)]_{d-n}) $, with
\begin{equation*}
[\Gamma^{\prime}_{R}(t)]_{j}=2\text{Im}\int_{\mathbb{R}^{d}}\theta^{\prime}\(\tfrac{z_{j}}{R}\)\partial_{j}u_{c} \overline{u_{c}}dx, \quad j\in\left\{1, 2, \ldots, d-n\right\}.
\end{equation*}
Since $G(u_{c}(t))=0$ for all $t\in \R$, we infer that
\begin{equation*}
\text{Im}\int_{|z_{j}|\leq R}\partial_{j}u_{c}\overline{u_{c}}dx=-\text{Im}\int_{|z_{j}|> R}\partial_{j}u_{c}\overline{u_{c}}dx.
\end{equation*}
By using the fact that $\theta^{\prime}(\tfrac{z_{j}}{R})=1$ for $|z_{j}|\leq R$, we conclude
\begin{equation*}
[\Gamma^{\prime}_{R}(t)]_{j}=-2\text{Im}\int_{|z_{j}|\geq
  R}\partial_{j}u_{c} \overline{u_{c}}dx+2\text{Im}\int_{|z_{j}|\geq
  R}\theta^{\prime}\(\tfrac{z_{j}}{R}\)\partial_{j}u_{c}
\overline{u_{c}}dx. 
\end{equation*}
This implies
\begin{equation}\label{Cdd}
|\Gamma^{\prime}_{R}(t)|\leq 10\int_{|z|\geq R}|\nabla u_{c}||u_{c}|dx\leq 5\int_{|z|\geq R}\left[|\nabla u_{c}|^{2}+|u_{c}|^{2}\right]dx.
\end{equation}
Combining Lemma \ref{compass} and \eqref{Cdd}, given $\eps>0$ (to be chosen later) there exists $R_{\eps}>0$ such that if $\tilde{R}_{k}:=R_{k}+R_{\eps}$, then 
\begin{equation}\label{Dq1}
|\Gamma^{\prime}_{\tilde{R}_{k}}(t)|\leq 5\eps.
\end{equation}
Moreover, by following the same argument as in the proof of \cite[Lemma 5.1]{DuyHolmerRoude2008} we get
\begin{align}\label{Dq2}
|\Gamma_{\tilde{R}_{k}}(0)|&\leq R_{\eps}\|\varphi\|^{2}_{L^{2}}+2\tilde{R}_{k}\eps,\\  \label{Dq3}
|\Gamma_{\tilde{R}_{k}}({t}^{\ast}_{k})|&\geq\tilde{R}_{k}(\|\varphi\|^{2}_{L^{2}}-3\eps)-2R_{\eps}\|\varphi\|^{2}_{L^{2}},
\end{align}
where ${t}^{\ast}_{k}=t_{0}(R_{k})$. Since $\tilde{R}_{k}\geq R_{k}\geq \alpha \tilde{t}_{k}$, combining the inequalities \eqref{Dq1}, \eqref{Dq2} and \eqref{Dq3} we infer that 
\begin{align*}
5\eps {t}^{\ast}_{k}&\geq \int^{{t}^{\ast}_{k}}_{0}|\Gamma^{\prime}_{\tilde{R}_{k}}({t})|\geq |\Gamma_{\tilde{R}_{k}}({t}^{\ast}_{k})- \Gamma_{\tilde{R}_{k}}(0)|\\
&\geq  {t}^{\ast}_{k}\alpha (\|\varphi\|^{2}_{L^{2}}-3\eps)-2R_{\eps}\|\varphi\|^{2}_{L^{2}},
\end{align*}
that is,
\begin{equation*}
{t}^{\ast}_{k}\left[\alpha\|\varphi\|^{2}_{L^{2}}-\eps(3\alpha+5)\right]\leq 2R_{\eps}\|\varphi\|^{2}_{L^{2}}.
\end{equation*}
By taking $\eps>0$ sufficiently small, letting ${t}^{\ast}_{k}\rightarrow\infty$  in the inequality above yields a contradiction. This proves the claim.\\
\textbf{Step 2.} There exits $\eta>0$ such that $P(u_{c}(t))\geq \eta$ for all $t\geq 0$. Indeed, if not, there exists a sequence of times $t_{k}$ such that
\begin{equation*}
P(u_{c}(t_{k}))<\frac{1}{k} \quad \text{for all $k$.}
\end{equation*}\
Since $\left\{u_{c}(t,\cdot,\cdot-z(t)); t\geq 0\right\}$ is precompact, there exists $f\in B_{1}$ such  that, passing to a subsequence if necessary, $g_{k}:=u_{c}(t_{k},\cdot,\cdot-z(t_{k}))\rightarrow f$ in $B_{1}$. Notice that   $S(f)=\lim_{k\rightarrow \infty}S(g_{k})=\tau_{c}<\beta$ and since $P(u_{c}(t_{k}))\geq 0$, it follows that $P(f)= \lim_{k\rightarrow \infty}P(g_{k})=0$. Thus,
$S(f)<\beta$ and $P(f)=0$. By Remark~\ref{Remarp}, we infer that $f=0$, which is a absurd because $S(f)=\tau_{c}>0$. \\ 
\textbf{Step 3.} Conclusion. We use the virial identities \eqref{LV1} and \eqref{LV3} with $u_{c}$ in place of $u$. We recall that 
\begin{equation}\label{f1}
V^{\prime\prime}(t)=4(d-n)P(u_{c}(t))+R_{1}+R_{2}+R_{3},
\end{equation}
where $R_{1}$, $R_{2}$ and $R_{3}$ are given by \eqref{R11}. Notice that there exists a constant $K$ independent of $t$ such that
\begin{equation}\label{f2}
|R_{1}+R_{2}+R_{3}+R_{4}|\leq K\int_{|z|\geq R}\left[|\nabla u_{c}(t)|^{2}+|u_{c}(t)|^{2} +|u_{c}(t)|^{2\sigma+2}\right]dx.
\end{equation}
By \eqref{LV1} it is clear that there exists a constant $L>0$ such that
\begin{equation}\label{f0}
|V^{\prime}(t)|\leq LR.
\end{equation}
From Lemma \ref{compass}, there exists $\rho> 1$ such that
\begin{equation}\label{f3}
\int_{|z+z(t)|\geq \rho}\left[|\nabla u_{c}(t)|^{2}+|u_{c}(t)|^{2} +|u_{c}(t)|^{2\sigma+2}\right]dx\leq \frac{2\eta(d-n)}{K},
\end{equation}
for every $t\geq 0$, where $\eta$ is given in Step 2. Moreover, by \eqref{Fsa} we obtain that there exists $t_{0}>0$ such that
\begin{equation}\label{f4}
|z(t)|\leq \frac{2\eta(d-n)}{4L}t \quad \text{for every $t\geq t_{0}$.}
\end{equation}
For $t^{\ast}>t_{0}$ we put
\begin{equation}\label{f5}
R_{t^{\ast}}=\rho+\frac{2\eta(d-n)}{4L}t^{\ast}.
\end{equation}
It is clear that $\left\{|z|\geq R_{t^{\ast}}\right\}\subset \left\{|z+z(t)|\geq\rho\right\}$ for all $t\in [t_{0}, t^{\ast}]$. Therefore, by 
\eqref{f2} and \eqref{f3} we get
\begin{equation}\label{f6}
|R_{1}+R_{2}+R_{3}+R_{4}|\leq {2\eta(d-n)}, \quad \text{for all $t\in [t_{0}, t^{\ast}]$.}
\end{equation}
Thus, by \eqref{f6} and Step 2 we have 
\begin{equation}\label{f7}
V^{\prime\prime}(t)\geq 2\eta(d-n)\quad \text{for all $t\in [t_{0}, t^{\ast}]$.}
\end{equation}
Integrating \eqref{f7} on $(t_{0}, t^{\ast})$, it follows from \eqref{f7} and \eqref{f0}
\begin{align*}
2\eta(d-n)(t^{\ast}-t_{0})&\leq\int^{t^{\ast}}_{t_{0}}V^{\prime\prime}(t)dt\leq |V^{\prime}(t^{\ast})-V^{\prime}(t_{0})|\leq 2LR_{t^{\ast}}\\
&=2L\rho+{\eta(d-n)t^{\ast}}.
\end{align*}
Choosing $t^{\ast}$ large enough, we get a contradiction. The proof of lemma is now completed.
\end{proof}

\begin{proof}[Proof of Theorem~\ref{Th1}~(i) (scattering result)]
The proof of scattering part of Theorem~\ref{Th1}  is an immediate consequence of the Proposition~\ref{P11ce} and Lemma~\ref{final}.
\end{proof}

\bibliographystyle{siam}
\bibliography{biblio}

\end{document}